\documentclass[11pt,english]{article}
\usepackage[utf8]{inputenc}
\usepackage[margin= 22mm,bottom=25mm, top= 22mm]{geometry}

\usepackage{textcomp}
\usepackage{graphicx}

\usepackage{amssymb,amsmath,amsthm,color,mathtools}
\usepackage{float}
\usepackage[font=small]{caption}
\usepackage[font=footnotesize]{subcaption}
\usepackage{tikz}
\usepackage[bookmarks=false,colorlinks,citecolor=blue,linkcolor=blue,urlcolor=blue]{hyperref}
\usepackage{enumitem}
\usepackage{thm-restate}
\usepackage{cleveref}

\usepackage{setspace}
\usepackage{caption}% http://ctan.org/pkg/caption
\newtheorem{theorem}{Theorem}
\newtheorem*{theorem*}{Theorem}
\newtheorem{lemma}[theorem]{Lemma}
\newtheorem*{lemma*}{Lemma}
\newtheorem{problem}[theorem]{Problem}
\newtheorem*{problem*}{Problem}

\newtheorem*{claim*}{Claim}
\newtheorem{proposition}[theorem]{Proposition}
\newtheorem*{proposition*}{Proposition}
\newtheorem{observation}[theorem]{Observation}
\newtheorem{corollary}[theorem]{Corollary}
\newtheorem*{corollary*}{Corollary}

\newtheorem{conjecture}[theorem]{Conjecture}

\newtheorem{innerclaim}{Claim}[theorem]

\newcommand{\xb}{\mathbf{x}}

\newcommand{\prm}{^{\prime}}

\newcommand{\dprm}{^{\prime \prime}}

\newcommand{\floor}[1]{\left \lfloor #1 \right \rfloor}
\newcommand{\ceil}[1]{\left \lceil #1 \right \rceil}

\newcommand{\abs}[1]{\left|#1\right|}

\newcommand{\nats}{\mathbb{N}}

\newcommand{\calk}{\mathcal{K}}
\newcommand{\expect}[1]{\mathbf{E}[#1]}

\newcommand{\prob}[1]{\mathbf{P}\big[#1\big]}

\newcommand{\e}{{\rm e}}

\newcommand{\calP}{\mathcal{P}}

\DeclareMathOperator{\SF}{SF}
\DeclareMathOperator{\KF}{KF}

\newcommand{\eps}{\varepsilon}
\newcommand{\calH}{\mathcal{H}}

\newcommand{\calG}{\mathcal{G}}

\newcommand{\calK}{\mathcal{K}}

\newcommand{\R}{\mathbb{R}}

\def \hG {\hat{G}}
\def \hS {\hat{S}}

\title{\vspace{-0.8cm}Ascending Subgraph Decomposition}
\author{
	Kyriakos Katsamaktsis\thanks{
		Department of Mathematics,
		University College London,
		Gower Street, London WC1E~6BT, UK.
		Email: \texttt{kyriakos.katsamaktsis.21}@\texttt{ucl.ac.uk}.
        Research supported by the Engineering and Physical Sciences Research Council [grant number EP/W523835/1].
	}
	\and
	Shoham Letzter\thanks{
		Department of Mathematics, 
		University College London, 
		Gower Street, London WC1E~6BT, UK. 
		Email: \texttt{s.letzter}@\texttt{ucl.ac.uk}. 
		Research supported by the Royal Society.
    }
    \and
    Alexey Pokrovskiy\thanks{
		Department of Mathematics, 
		University College London, 
		Gower Street, London WC1E~6BT, UK. 
		Email: \texttt{dralexeypokrovskiy}@\texttt{gmail.com}.
	}
    \and
    Benny Sudakov\thanks{
		Department of Mathematics, ETH Z\"urich, Switzerland
		Email: \texttt{benjamin.sudakov}@\texttt{math.ethz.ch}. Research supported in part by SNSF grant 200021\_196965.
    }
}
\date{}
\begin{document}

\maketitle
\begin{abstract}
	A typical theme for many well-known decomposition problems is to show that some obvious necessary conditions for decomposing a graph $G$ into copies $H_1, \ldots, H_m$
	are also sufficient. One such problem was posed in 1987, by Alavi, Boals, Chartrand, Erd\H{o}s, and Oellerman. They conjectured that the edges of every graph with $\binom{m+1}2$ edges can be decomposed into subgraphs $H_1, \dots, H_m$ such that each $H_i$ has $i$ edges and is isomorphic to a subgraph of $H_{i+1}$. In this paper we prove this conjecture for sufficiently large $m$.
\end{abstract}

\section{Introduction}

For a graph \(G\), we say a collection of graphs \(H_1,\hdots, H_m\)
is a  \emph{decomposition} of \(G\), if  \(G\) is an edge-disjoint union of \(H_1,\hdots, H_m\).
Decomposition problems have been a central theme in combinatorics since Euler's work on the existence of orthogonal Latin squares in the $18^{\text{th}}$ century; recall that a \emph{Latin square} is an $n \times n$ array, filled with numbers from $[n]$, such that each $i \in [n]$ appears exactly once in each row and column. Euler asked for which values of $n$ there exist two $n\times n$ Latin squares $L, L'$ with the property that all $n^2$ ordered pairs $(L_{i,j}, L'_{i,j})$, with $1\leq i, j\leq n$, are distinct. This problem turns out to have an equivalent formulation in terms of graph decompositions. Indeed, one can show that Euler's problem is  equivalent to determining which complete 4-partite graphs $K_{n,n,n,n}$ have a decomposition \(H_1,\hdots, H_m\) where each $H_i$ is a $K_4$-factor, namely each $H_i$ is a collection of vertex-disjoint $K_4$'s, such that each vertex of $K_{n,n,n,n}$ appears in one $K_4$. 
        
A large variety of other graph/hypergraph decomposition problems has been studied over the years. A typical theme for many well-known such problems is to show that some obvious necessary conditions for decomposing a graph $G$ into copies of $H_1,\hdots, H_m$ are also sufficient. For example, the famous ``existence of designs'' question posed in 1853 by Steiner asked to prove that for large enough $n$, the complete $r$-uniform hypergraph $\mathcal K^{(r)}_n$ has a decomposition into copies of $\mathcal K^{(r)}_k$ if and only if $\binom{n}r$ is divisible by $\binom{k}r$ (which is equivalent to asking that the number of edges of $\mathcal K^{(r)}_n$ is divisible by the number of edges of $\mathcal K^{(r)}_k$) and also that, for \(i\in [r-1]\), $\binom{n-i}{r-i}$ is divisible by $\binom{k-i}{r-i}$ (which is equivalent to asking that the codegree of any  \(i\)-set of vertices in  $\mathcal K^{(r)}_n$ is divisible by the codegree of each \(i\)-set of vertices in  $\mathcal K^{(r)}_k$).
The  existence of designs problem was solved by Keevash in~\cite{keevash2014existence} (see~\cite{glock2016existence} for an alternative proof).  
Other recently solved problems include the Oberwolfach problem (decompositions of complete graphs into cycle factors, see \cite{oberwolfachglock, oberwolfachkeevash}) and
Ringel's conjecture (decompositions of complete graphs into copies of a fixed tree $T$, see \cite{ringel,ringelkeevash}).

One famous conjecture in the area that is still open and the closest to the problem we will consider here is the tree packing conjecture of Gy\'arf\'as~\cite{gyarfas}.
It says that for  any collection of trees \(T_1,\hdots,T_{n}\) where \(T_i\) has \(i\) edges, the complete graph $K_{n+1}$ can be decomposed into copies of $T_1, \dots, T_n$. Again the motivation here is to show that the trivial condition $\binom {n+1}2=e(K_n)=e(T_1)+\dots+e(T_n)=1+\dots+n$ is also sufficient for such a decomposition to exist.
Some strong results have been proved for this problem when there is some control over the degrees of $T_i$. Joos, Kim, K\"uhn, and Osthus \cite{joos2019optimal} proved the conjecture when  $\Delta(T_i)\leq \Delta$ for all $T_i$ and $n$ is large compared to $\Delta$, and subsequently  Allen, B\"ottcher, Clemens, Hladk\'y, Piguet, and Taraz \cite{allen2021} proved the conjecture when $\Delta(T_i)\leq cn/\log n$ for all $T_i$, for some universal constant $c > 0$.
But for trees with unbounded degrees, the conjecture is still wide open. For example, it is not even known if we can find edge-disjoint copies of $T_{n}, T_{n-1}, \dots, T_{n-5}$ in $K_n$.

All the problems discussed so far have had the host graph $G$ being complete and the target graphs $H_1, \dots, H_m$ being similar to each other in some way (i.e.\ we wanted all $H_i$ to be copies of $\mathcal K_k^{(r)}$, or all $H_i$ to be cycle factors, or all $H_i$ to be trees). In 1987,
Alavi, Boals, Chartrand, Erd\H{o}s, and Oellerman suggested that some degree of similarity of $H_1, \dots, H_m$ can be still achieved \emph{without putting any additional restrictions on $G$ whatsoever} aside from the trivial condition that $e(G)=e(H_1)+\dots+e(H_m)$. Specifically, they called a decomposition of $G$ into $H_1, \dots, H_m$ \emph{ascending} if $e(H_i)=i$, and each $H_i$ is isomorphic to a subgraph of $H_{i+1}$. Since $e(H_1)+\dots+e(H_m)=\binom{m+1}2$, the trivial necessary condition for the existence of an ascending decomposition is $e(G)=\binom{m+1}2$. Alavi, Boals, Chartrand, Erd\H{o}s, and Oellerman \cite{original87} conjectured that this is also sufficient.
\begin{conjecture}[Alavi, Boals, Chartrand, Erd\H{o}s, and Oellerman \cite{original87}]
\label{Conjecture_main}
    Every graph $G$ with $\binom{m+1}2$ edges has an ascending subgraph decomposition, namely a decomposition $H_1, \ldots, H_m$ such that $e(H_i) = i$ and $H_i$ is a subgraph of $H_{i+1}$.
\end{conjecture}
This conjecture does not prescribe the graphs $H_1, \dots, H_m$ in the decomposition as much as the conjectures of, e.g., Gy\'arf\'as or Ringel do. However, this is necessary if one wants to prove a decomposition statement that holds \emph{for all graphs}. Indeed, if $G$ is a matching, then the only subgraphs it has are matchings and therefore in any decomposition, all $H_i$'s must be matchings. Similarly, if $G$ is a star, then in any decomposition, all $H_i$'s must be stars. So, in order for a decomposition result to hold for all possible host graphs $G$, the result must allow for using different $H_i$'s for different host graphs.
        
There are several partial results that find an ascending subgraph decomposition when 
\(G\) lies in a restricted class of graphs.
In~\cite{original87} the authors verified the conjecture when \(G\) has maximum degree at most \(2\), using only matchings in the decomposition.
After partial results~\cite{original87,faudree87,faudree88}
Faudree and Gould~\cite{faudree90} proved that forests have an ascending subgraph decomposition into star forests.
The case of regular graphs was settled by Fu and Hu~\cite{fu2002regular} using linear forests in the decomposition.
Faudree, Gould, Jacobson, and Lesniak~\cite{faudree88} proved the conjecture for \(m\) sufficiently large and \(G\) of maximum degree at most
\((2\sqrt{2}-2)m\), where the graphs in the decomposition are short paths. This extends a previous result of Fu~\cite{fu1990note} who proved the conjecture for graphs of degree at most \(m/2\).

Finding ascending decompositions can be difficult and interesting even for rather specific host graphs. For example
Ma, Zhou, and Zhou~\cite{ma-zhou-zhou} proved that a star forest with components of size
at least \(m\) has an ascending subgraph decomposition into stars.
Here and throughout the paper, the \emph{size} of a graph is the number of its edges.
This was previously another conjecture of Alavi, Boals, Chartrand, Erd\H{o}s, and Oellerman which is appealing due to having an entirely number-theoretic formulation. In this case the problem is equivalent to saying that for any set of numbers $a_1, \dots, a_t\geq m$ with $a_1+\dots+a_t=m(m+1)/2$ it is possible to decompose the interval $[m]$ into sets $A_1, \dots, A_t$ with the numbers in each $A_i$ summing to $a_i$. For a comprehensive survey of prior results in this area we refer the reader to~\cite[Chapter 8]{survey}.

Our main theorem resolves the ascending subgraph decomposition conjecture for all large enough graphs.
\begin{theorem} \label{thm:main}
    Let $m$ be a sufficiently large integer. Then every graph with $\binom{m+1}{2}$ edges has an ascending subgraph decomposition.
\end{theorem}

As an intermediate step towards proving this theorem, we show if $G$ is a star forest where the $i^{\text{th}}$ star has size at least $\min\{1600i, 20(m+1)\}$, then $G$ has an ascending subgraph decomposition into stars. This statement is more flexible than the one by Ma, Zhou, and Zhou~\cite{ma-zhou-zhou} (mentioned above), since it allows for initial stars to be small. This result also has a purely number-theoretic formulation.
\begin{theorem}\label{Theorem_stars_introduction}
    For any set of numbers $a_1, \dots, a_t$ with $a_1+\dots+a_t=m(m+1)/2$ and each $a_i\geq \min\{1600i, 20(m+1)\}$, it is possible to decompose the interval $[m]$ into sets $A_1, \dots, A_t$ with the numbers in each $A_i$ summing to $a_i$.
\end{theorem}
 
We also prove the following statement that may be of independent interest.  It says that any graph \(G\) with $\Theta(m^2)$ edges and maximum degree $O(m)$ can be decomposed into \(\Theta(m)\) pairwise isomorphic graphs plus \(o(m^2)\) edges; see \Cref{lem:approx-isomorphic-decomposition}.

\paragraph{Notation.}
We use standard asymptotic notation throughout. For positive real functions \(f,g\) of a positive variable \(n\) we write \(f = O(g)\) if the limit $\limsup_{n \to \infty} f(n)/g(n)$ is finite,
and write \(f=o(g)\) if the limit is \(0\).

\Cref{thm:main} follows by combining \Cref{Theorem_stars_introduction} above with \Cref{lem:max-deg-linear} below, which shows that a graph 
with \(\binom{m+1}{2}\) edges and maximum degree \(O(m)\) has an ascending subgraph decomposition (we sometimes abbreviate this to ASD).
Before describing how we deal with each component separately, let us sketch how \Cref{thm:main} follows from \Cref{Theorem_stars_introduction} and \Cref{lem:max-deg-linear}. First note that \Cref{Theorem_stars_introduction} implies the following.
\begin{theorem} \label{thm:main-stars}
    Let $G$ be an edge-disjoint union of stars with $\binom{m+1}{2}$ edges, where the $i^{\text{th}}$ star has size at least $\min\{1600i, 20(m+1)\}$. Then $G$ has an ascending subgraph decomposition into stars.
\end{theorem}

\paragraph{Proof of \Cref{thm:main} using \Cref{thm:main-stars} and \Cref{lem:max-deg-linear}.}
Suppose \(G\) is an arbitrary graph with \(\binom{m+1}{2}\) edges.
Set \(G_0 = G\), and repeat the following: for \(i\ge 1\) let \(v_i \in V(G_{i-1})\) be a vertex of degree at least \(\Omega\left(\sqrt{e(G_{i-1}})\right)\), if it exists, and let \(G_i := G_{i-1} \setminus \{v_i\}\).
We can continue this process until  we reach a graph \(G'\) of maximum degree 
\(O\left(\sqrt{e(G')}\right)\), which by \Cref{lem:max-deg-linear} has an ASD denoted \(H_1,\hdots, H_k\) for $k\approx \sqrt{2e(G')}$.
A technicality here is that \(e(G')\) might not be a binomial coefficient. 
However, our argument gives an ASD also for such graphs, for a natural generalisation of an ASD (see the beginning of \Cref{subsec:reduction} for the definition).
For now let us ignore this technicality and assume that $e(G')$ is a binomial coefficient, i.e.\ $e(G') = \binom{k+1}{2}$ for some integer $k$.
Let \(G\dprm = G \setminus G'\).
Then \(G\dprm\) is an edge-disjoint union of large stars, 
which readily implies that inside \(G\dprm\) we can find isomorphic stars
\(
\hS_1,\hdots, \hS_k
\)
of size $m-k$
such that \(\hS_i\) is vertex-disjoint of \(H_i\).
Then we set \(\hG_i = H_i \cup \hS_i\),
and observe that \(\hG_{i}\) is isomorphic to a subgraph of \(\hG_{i+1}\).
The graphs \(\hG_1,\hdots, \hG_k\) will be the last \(k\) graphs in the ASD of \(G\).
Finally, by taking some extra care when picking the stars $\hS_i$, we may assume that \(G\dprm - \bigcup_i \hS_i\) is still an edge-disjoint union of large stars, which by \Cref{thm:main-stars} has an ASD into stars $S_1, \ldots, S_{m-k}$.
These stars will be the first \(m-k\) graphs in the ASD of \(G\), and along with \(\hG_1,\hdots, \hG_k\) they yield a complete ASD of \(G\).

We prove \Cref{thm:main} in \Cref{sec:reduction}. We next sketch the other two main parts of the proof.

\paragraph{\Cref{thm:main-stars}.} 
The proof of this is easiest to explain in the number theoretic formulation given in \Cref{Theorem_stars_introduction}.
If we assume that \(m\) is even (the odd case reduces to the even case),
then
\(\sum_i a_i = \frac{m}{2} (m+1)\) is divisible by \(m+1\).
Suppose first that each \(a_i\) is divisible by \(m+1\), i.e.\ 
\(a_i = \lambda_i (m+1)\) for some positive integer \(\lambda_i\), so 
\(\sum_{i} \lambda_i = m/2\).
Then 
the sets \(\{x, m+1-x\}\), where 
\(x\in [m/2]\), partition \([m]\),
and any \(\lambda_i\) of them sum to \(a_i\).
Hence  we can set \(
A_i\) to be any \(\lambda_i\) of these pairs, such that each pair is used by exactly one $A_i$.
We reduce the general case to the above setup by iterating the following procedure: note that since $\sum a_i\equiv 0 \pmod {m+1}$, there cannot be just one $a_i$ that is not divisible by $m+1$ --- so we have distinct $a_i$ and $a_j$ which are not divisible by $m+1$. Pick   \(x\neq y \in [m/2]\), with \(a_i \equiv x+y \pmod{m+1}\). Replace $a_i$ by $a_i':=a_i-x-y$, $a_j$ by $a_j':=a_j-(m+1-x) -(m+1-y)$ and remove the elements $x, m+1-x, y, m+1-y$ from $[m]$. Note that we need here to guarantee $a_i', a_j'>0$. This is done by appropriately choosing $x, y$ and using the condition of the theorem on the size of the $i^{\text{th}}$ component of a star forest.
The effect is that we have reduced the number of terms not divisible by $m+1$, and hence, by repeating this procedure we end up in the situation when all $a_i$ are divisible by $m+1$, which we already know how to solve. \Cref{thm:main-stars} is proved in \Cref{subsec:number-theory}.

\paragraph{\Cref{lem:max-deg-linear}.} Now we sketch the proof that every \(G\) with maximum degree at most \(O(m)\) has an ASD.
The main idea is to almost decompose \(G\) in several stages, so that at each stage the decomposition consists of ``nice'' graphs which at the very end can be combined to form an ASD.
First, we decompose edges incident to small degree vertices (namely, at most $cm$ for some appropriate constant $c$) into isomorphic star forests 
and a remainder that has few edges (cf.\ \Cref{lem:approx-star-forest-decomposition}).
For this step, we first decompose edges with one large degree vertex and one small degree vertex into isomorphic star forests, via \Cref{lem:star-forests}, then we decompose edges incident to only small degree vertices via Vizing's theorem and \Cref{lem:balancing-matching}, and finally we combine the star forests and matchings via \Cref{lem:large-matching-forest}.
Second, we almost decompose the edges incident only to large degree vertices into copies of complete bipartite graphs (cf.\ \Cref{lem:Ktt-decomposition}). 
This gives an almost decomposition of \(G\) into isomorphic ``forests'' whose components are stars and complete bipartite graphs, and each forest contains a large matching (cf.\ \Cref{lem:approx-isomorphic-decomposition}).
Third, by carefully rearranging the graphs in the decomposition, we obtain an ``approximate'' ASD consisting of a remainder \(R\) of small maximum degree; and graphs
\((H_1,\hdots, H_{m'})\), where each \(H_i\) has a large isolated matching (i.e.\ a matching touching no other edges of $H_i$), and \(H_i\) is isomorphic to a subgraph of \(H_{i+1}\)
(cf.\ \Cref{lem:approx-isomorphic-decomposition-stronger}). 
In this step it is crucial that we are working with a graph having maximum degree $O(m)$. If this were not the case, then the graph need not have any large matchings at all.
These will be the basis for the last \(m'\) graphs in the ASD of \(G\).
Fourth, in \Cref{lem:approx-ascending-decomposition}, we randomly remove an isolated matching 
\(M_i\) of each \(H_i\) so that \(H_i\setminus M_i\) has the correct number of edges for its position in the ASD.
Let \(F= \bigcup_i M_i\).
Then from a standard concentration bound (Chernoff's bound) it follows that each vertex has small degree in \(F\), and hence the maximum degree of \(F\cup R\) is small.
Therefore we can find an ASD of \(F\cup R\) into matchings by using a result of Fu~\cite{fu1990note} about ascending subgraphs decompositions into matchings of graphs with small maximum degree (cf.\ \Cref{lem:asd-bounded-degree}).
By construction, the graphs \(H_i\setminus M_i\) contain a large matching, so that the ASD of \(F\cup R\) into matchings combined with \(H_1\setminus M_1, \hdots, H_{m\prm} \setminus M_{m\prm}\) give an ASD of \(G\).
The proof of \Cref{lem:max-deg-linear} is given at the end of \Cref{sec:bounded-degree}.

\section{Finding ascending subgraph decompositions} \label{sec:reduction}

In  \cref{subsec:number-theory} we prove Theorem~\ref{Theorem_stars_introduction} about decomposing the interval $[m]$ into sets summing to $a_1, \ldots, a_k$, for any such sequence with appropriate properties.
In \cref{subsec:reduction} we use it to reduce Theorem~\ref{thm:main} to graphs with linear maximum degree.

\subsection{Ascending star decompositions} \label{subsec:number-theory}
	The goal of this section is to prove \Cref{Theorem_stars_introduction} (and \Cref{thm:main-stars} which immediately follows from it). We now introduce some notation.   
    Given a sequence of positive integers
    \(a_1,\hdots,a_k \),
    we say \(R\subseteq \nats\) \emph{separates} \(a_1,\hdots, a_k\) if there exists a partition
    \(I_1,\hdots, I_k\) of \(R\)
    such that 
    \(a_i = \sum_{x\in I_i} x\).
	The next lemma is an equivalent formulation of \Cref{Theorem_stars_introduction}.    
	\begin{lemma}[Equivalent formulation of Theorem~\ref{Theorem_stars_introduction}] \label{lem:number-theory}
        Let \(k,m\) be positive integers.
        Let $a_1\leq \ldots\leq a_k$ be a sequence of positive integers such that $\sum_i a_i = \binom{m+1}{2}$, and 
		$a_i \ge \min \{ 1600 i, 20(m+1)\}$ for $i \in [k]$. Then \([m]\) separates \(a_1,\hdots, a_k\).
    \end{lemma}
    \begin{proof}
    Let $k'$ be maximal such that $a_{k'} < 20(m+1)$. Then
    \[
        \binom{m+1}{2} \ge \sum_{i=1}^{k'}a_i\geq \sum_{i = 1}^{k'} 1600i = 1600 \binom{k'+1}{2} \ge \binom{40k'+1}{2},
    \]
    so \(k'\le m/40\).
    Additionally, $k - k' \le \binom{m+1}{2} / 20(m+1) = m/40$. Altogether, $k \le m/20$. 
    This also shows that $a_k\ge 20(m+1)$ (otherwise, $k = k' \le m/40$ and then $a_{k'} \ge \binom{m+1}{2}/k' \ge 20(m+1)$, a contradiction).
    
    If $m$ is odd, we set \(m\prm := m-1\) and 
	\(a_k\prm := a_k-m \ge 19(m\prm+1)\), so that \(a_1,\hdots,a_{k-1},a_k\prm\) is a sequence of positive integers summing to \(\binom{m\prm +1}{2}\) such that the \(i^{\text{th}}\) term is at least \(\min \{ 1600 i, 19(m\prm+1)\}\), and then it suffices to show \([m\prm]\) separates \(a_1,\hdots,a_{k-1},a_k\prm\).
    In this case we have $k \le (m'+1)/20 \le m'/16$ (using $m' \ge 4$, which follows implicitly from the assumptions).
    
    Thus, from now on we assume that $m$ is even, $a_i \ge \min\{1600i, 19(m+1)\}$ for $i \in [k]$, and $k \le m/16$.

    Let
    \(P_m := \left\{ \{x, m+1-x\}: x \in [m/2] \right\} \).
    For \(S\subseteq P_m\), we say that \(S\) \emph{separates} \(a_1,\hdots, a_k\) if there exists \(S\prm \subseteq S\) such that \(\bigcup S\prm\) separates \(a_1,\hdots, a_k\);
    if \(S\prm = S\) we say \(S\) separates the sequence \emph{perfectly}.
    
    \begin{innerclaim} \label{claim:ASD-forest:all-divisible}
		Let $a_1, \ldots, a_k$ be a sequence of positive integers, such that \(a_i = \lambda_i( m+1)\) for some positive integer $\lambda_i$, for every $i \in [k]$.
         Let
        \(S\subseteq P_m \)
        with 
        \(\abs{S}  \ge \sum_i \lambda_i\).
        Then $S$ separates $a_1, \ldots, a_k$.
    \end{innerclaim}
    \begin{proof}
        Since $\abs{S} \ge \sum_i \lambda_i$,
        for each $i \in [k]$ we can pick a set $S_i$ consisting of \(\lambda_i\) distinct sets \(\{x, m+1-x\}\) from \(S\),  so that the sets $S_i$ are pairwise disjoint. The elements in $S_i$ sum to $a_i$, for $i \in [k]$.
    \end{proof}
     
    For \(n \in \nats\) let
    \begin{equation*}
		T(n) = \left\{ (x,y) : 1 \le x < y \le m,\,\, x+y \equiv n \!\!\!\pmod{m+1} \text{ and } x+y \leq n \right\}.
    \end{equation*}
    \begin{innerclaim}\label{claim:Tnbound}
        $|T(n)| \ge \min\{n/2, m/2\} - 1$.
    \end{innerclaim}
    \begin{proof}
		Consider first the case $n \ge m+1$. For $x \in [m]$, take $y_x$ to be the smallest non-negative number such that $x + y_x \equiv n \pmod{m+1}$. Then $y_x \in [0,m]$, and, since $n - x > 0$, we have $y_x \le n - x$, showing $x + y_x \le n$. Thus the number of ordered pairs $(x,y)$ satisfying $x \in [m]$, $y \in [0,m]$, $x+y \equiv n\pmod{m+1}$, and $x + y \le n$ is at least $m$. Note that at most one such pair satisfies $x = y$ (using that $m$ is even and so $m+1$ is odd), and at most one pair has $y = 0$. Hence there are at least $m/2 - 1$ pairs $(x,y)$ with $x, y \in [m]$, $x + y \equiv n \pmod{m+1}$ and $x < y$. This shows $|T(n)| \ge m/2-1$.
        
        If $n \le m$, for each $x \in \big[\floor{(n-1)/2}\big]$ taking $y_x := n - x$ shows that $|T(n)| \ge n/2 - 1$.
    \end{proof}
     
    \begin{innerclaim} \label{claim:ASD-forest:all-large}
        Let 
        \(a_1,\hdots, a_k\) be a sequence of positive integers, satisfying \(a_i \ge 3(m+1)\) for $i \in [k-1]$, $a_k \ge m+1$, and $\sum_i a_i = \ell(m+1)$ for an integer $\ell$.
        Let 
        \(S \subseteq P_m\) satisfy \(\abs{S}\ge m/4 +2k\).
        Then there is a sequence $b_1, \ldots, b_k$ such that $a_i \ge b_i$ and $a_i \equiv b_i \pmod{m+1}$, for $i \in [k]$, which is separated by $S$.
    \end{innerclaim}

    \begin{proof}
        We prove the statement by induction on $k$.
		For the base case $k = 1$, notice that $a_1 \equiv 0 \pmod{m+1}$ and thus we can take $b_1 = m+1$ (and use any $S'\subseteq S$ of size $1$, which separates $b_1$ because $S\subseteq P_m$).

        For the induction step, assume that $k \ge 2$ and that the claim holds up to \(k-1\).
        Since $a_k \ge m+1$, \Cref{claim:Tnbound} tells us that $|T(a_k)| \ge m/2 - 1$. 
        Because every $x \in [m]$ appears in at most one pair in $T(a_k)$, and $\bigcup S\subseteq \bigcup P_m=[m]$ the number of pairs in $T(a_k)$ containing an element not in $\bigcup S$ is at most $m - |\bigcup S| = m - 2|S| \le m/2 - 4k < m/2-1$, showing that there is a pair $(x,y) \in T(a_k)$ with $x,y \in \bigcup S$.
        
        Define $b_k := x + y$ and let 
        \begin{equation*}
            b_{k-1}'' := \left\{
                \begin{array}{ll}
                    (m+1 - x) + (m+1 - y) & \text{if $x+y \neq m+1$}, \\
                    0 & \text{otherwise}.
                \end{array}
                \right.
        \end{equation*}
        Set $S' := S \setminus \left\{\{x, m+1-x\}, \{y, m+1-y\}\right\}$, and let $a_i' := a_i$ for $i \in [k-2]$ and $a_{k-1}' := a_{k-1} - b_{k-1}''$.
        Now apply the induction hypothesis to the sequence $a_1', \ldots, a_{k-1}'$. To see that it is applicable, notice that $a_i \ge 3(m+1)$ for $i \in [k-2]$ and $a_{k-1}' \ge a_{k-1} - 2(m+1) \ge m+1$. 
        Moreover, we have 
        \(
        \sum_{i=1}^{k'-1} a_i ' = \ell (m+1) - a_k - b\dprm_{k-1},
        \)
        which is divisible by \(m+1\).
        Finally $|S'| = \abs{S}-2 \ge m/4 + 2(k-1)$. 
        Hence, by induction, there is a sequence $b_1', \ldots, b_{k-1}'$ which is separated by $S'$ and which satisfies $b_i' \le a_i'$ and $a_i' \equiv b_i' \pmod{m+1}$. Set $b_i := b_i'$ for $i \in [k-2]$, $b_{k-1} := b_{k-1}' + b_{k-1}''$ and recall that $b_k=x+y$.
        Then the sequence \(b_1,\hdots, b_k\) is separated by \(S\), which proves the induction step.
     \end{proof}
     
    \begin{innerclaim} \label{claim:ASD-forest:combine}
        Let $a_1, \ldots, a_k$ be a sequence of positive integers with
        \(a_i \ge 3(m+1)\) and
        \(\sum_i a_i = \ell (m+1)\), for some integer \(\ell\).
        Let \(S\subseteq P_m\) with 
        \(\abs{S} \geq \max \{ \ell,\, m/4 +2k\}\).
        Then \(S\) separates \(a_1,\hdots,a_k\).
    \end{innerclaim}
    \begin{proof}
        Let $b_1, \ldots, b_k$ be a sequence as in \Cref{claim:ASD-forest:all-large}, and let $S' \subseteq S$ be a set that separates this sequence perfectly. Since pairs in $S$ add up to $m+1$, this tells us that $|S'|=\frac{1}{m+1}\sum_{x \in S'} x=\frac{1}{m+1}\sum_ib_i$. 
        Using that $a_i - b_i$ is divisible by $m+1$, we can write $\sum_i(a_i - b_i) = \ell' (m+1)$ for an integer $\ell'$. 
        Since $\abs{S} \ge \ell$ we have 
        $|S \setminus S'|=|S|-|S'|=|S|- \frac{1}{m+1}\sum_ib_i =|S|+\ell'-\frac1{m+1}\sum_ia_i=|S|+\ell'-\ell\ge \ell'$, and thus Claim~\ref{claim:ASD-forest:all-divisible} shows that $S \setminus S'$ separates the sequence $a_1 - b_1, \ldots, a_k - b_k$.
        Hence \(S\) separates \(a_1,\hdots, a_k\), as desired. Indeed, since $S'$ separates $b_1, \ldots, b_k$, we have disjoint subsets $I_i\subseteq \bigcup S'$ with $\sum I_i=b_i$. Since $S\setminus S'$ separates $a_1-b_1, \ldots, a_k-b_k$, we have disjoint subsets $J_i\subseteq \bigcup S\setminus S'$ with $\sum J_i=a_i-b_i$. Now the sets $I_i\cup J_i\subseteq \bigcup S$ satisfy the definition of $S$ separating $a_1, \ldots, a_k$.
    \end{proof}
     
     \begin{innerclaim} \label{claim:ASD-forest:not-too-large}
        Let \(a_1 \le \hdots \le a_k\) be a sequence such that
		\(a_j \ge \min \{16 j, 3(m+1)\}\) for every $j \in \{i, \ldots, k\}$, and let \(S\subseteq P_m\).
        Let \(i\in[k]\) and \(\ell\) be an integer such that
        \(\sum_{j=i}^k a_j = \ell (m+1)\).
        Assume that
        \(\abs{S}\geq \ell\), and
        \(\ell \ge \max \{ m/2-4i +1, m/4+4(k-i), 5(k-i+1)\}\).
        Then \(S\) separates \(a_i,\hdots, a_k\).
    \end{innerclaim}
    \begin{proof}
        We prove the claim by induction.
        For the base case \(i=k\) we have
        \(a_k = \ell(m+1),\)
        and then any \(S\subseteq P_m\) with
        \(\abs{S}\ge \ell\) separates \(a_k\).
        
        For \(i<k\), if \(a_i \ge 3(m+1)\) the claim follows from Claim~\ref{claim:ASD-forest:combine} (using the assumption $|S| \ge m/4 + 4(k-i)$, since the length of the sequence is \(k-i\)), so we may assume otherwise. Thus $|T(a_i)| \ge \min\{a_i/2, m/2\} - 1 \ge 8i - 1$ (using \Cref{claim:Tnbound} for the first inequality, and $a_i\geq 16i$ and $i\leq k\leq m/16$ for the second inequality). Notice that every $x \in [m]$ is in at most one pair in $T(a_i)$. Hence, the number of pairs $(x,y) \in T(a_i)$ containing an element from $[m] \setminus \bigcup S$ is at most $m - |\bigcup S| = m - 2|S| \le m - 2\ell \le 8i-2$. It follows that there is a pair $(x,y) \in T(a_i)$ with $x,y \in \bigcup S$.
        
        Define $a_i' := a_i - x - y$. If $x+y \neq m+1$, define $a_k' := a_k - (m+1 - x) - (m+1 - y)$ (and otherwise set $a_k' := a_k$).
        Notice that
        \( a_k \geq \frac{\ell (m+1)}{k-i+1} \ge 5(m+1) \),
        so $a_k' \ge 3(m+1)$.
        
        Moreover, since 
        \(a_i' \in \{0,m+1, 2(m+1) \}\),
		we can separate $a_i'$ 
		by using (at most) two pairs from \(S \setminus \left\{\{x, m+1-x \}, \{y, m+1-y \}\right\} \).
		Let \(S\prm\) be the remainder of \(S \setminus \{\{x, m+1-x\}, \{y, m+1-y\}\}\).
        
		Let $b_{i+1}, \ldots, b_k$ be the non-decreasing sequence obtained by permuting the elements $a_{i+1}, \ldots, a_{k-1}, a_k'$. Then 
        $b_j \ge \min\{16j, 3(m+1)\}$ for all $j \in [i+1, k]$ (let $t$ be such that $b_t=a_k'$. For $j\geq t$ we have $b_j\geq b_t=a_k'\geq 3(m+1)$. For $j<t$ we have $b_j=a_j\geq \min(16j, 3(m+1))$).
        We will now show that \(S\prm\) separates
        $b_{i+1}, \ldots, b_k$, implying that the original sequence $a_i, \ldots, a_k$ is separated by \(S\).
        
        Observe by the definition of \(S\prm\) that 
        $\sum_{i+1 \le j \le k}b_j = \ell'(m+1)$ with $\ell'$ an integer such that
        \(|S'| \ge \ell\prm \ge \ell-4\).
        We thus have
        \begin{align*}
            \ell' \ge \left\{
            \begin{array}{l}
                 m/2 - 4i + 1 - 4 = m/2 - 4(i+1) + 1 \\
                 m/4 + 4(k - i) - 4 = m/4 + 4(k-(i+1))\\
                 5(k - i + 1) - 4 \ge 5(k - i).
            \end{array}
            \right.
        \end{align*}
        Hence the conditions of the induction hypothesis hold and, by induction, \(S\prm\) separates $b_{i+1}, \ldots, b_k$, as required.
    \end{proof}
    
    We now prove the lemma by verifying the conditions of Claim~\ref{claim:ASD-forest:not-too-large} for \(i=1\), \(\ell=m/2\) and \(S=P_m\). 
    Recalling that $k \le m/16$, we have $\ell = m/2 \ge \max\{m/2-4\cdot 1+1, m/4 + 4(k-1), 5k\}$, as required for the claim.
    Hence by Claim~\ref{claim:ASD-forest:not-too-large} the theorem follows.
    \end{proof}
   
\subsection{Proof of main result} \label{subsec:reduction}
    
	An \emph{ascending subgraph decomposition} of a graph $G$ with $e$ edges, where $\binom{m}{2} < e \le \binom{m+1}{2}$, is a decomposition of $G$ into graphs $H_1, \ldots, H_m$, such that $H_i$ is isomorphic to a subgraph of $H_{i+1}$ and $e(H_i) \le e(H_{i+1}) \le e(H_i) + 1$.
	Specifically, writing $t = e - \binom{m}{2}$ (so $1 \le t \le m$), we have 
	\begin{equation*}
		e(H_i) = \left\{
		\begin{array}{ll}
			i & i \le t \\
			i-1 & i > t.
		\end{array}
		\right.
	\end{equation*}

	In the next section we will prove the following lemma, showing that graphs with roughly $m^2/2$ edges and maximum degree $O(m)$ have an ASD.

	\begin{restatable}{lemma}{lemmaMaxDegLinear} \label{lem:max-deg-linear}
		Let $c=10^6$ and \(m\) sufficiently large.
		Suppose that $G$ is a graph satisfying $e(G) \in (\binom{m}{2}, \binom{m+1}{2}]$ and $\Delta(G) \le cm$. Then $G$ has an ascending subgraph decomposition.
	\end{restatable}

	We use this lemma to prove our main result, Theorem~\ref{thm:main}.
    
    \begin{proof}[Proof of Theorem~\ref{thm:main} using Lemma~\ref{lem:max-deg-linear}]
		Let $c = 10^6$ (as in \Cref{lem:max-deg-linear}), let $m_0$ be such that Lemma~\ref{lem:max-deg-linear} holds for $m \ge m_0$,
        and let \(m\ge m_0^2\).
  
		Let $G_0=G$ and let $v_1, \ldots, v_k \in V(G)$ be such that
        $v_{i+1}$ has maximum degree in $G_i := G - \{v_1, \ldots, v_i\}$ and satisfies $d_{i+1} := d_{G_i}(v_{i+1}) > c \sqrt{e(G_i)}$, for $i \in [0, k-1]$. Suppose that this is a maximal sequence with this property i.e.\ that $\Delta(G_k) \le c\sqrt{e(G_k)}$.
		We may assume \(k\ge 1\) since otherwise we are done by Lemma~\ref{lem:max-deg-linear}.
        Note that $d_1\geq \dots \geq d_k$,
        \(d_1 \ge c \sqrt{\binom{m+1}{2}} \ge m\),
        and \(e(G_{k-1})\ge 1\) since otherwise \(\Delta(G_{k-1})\le \sqrt{e(G_{k-1})}\) and we would get a sequence of length \(k-1\), contradicting maximality.

        We claim that the sequence 
        \(d_1, \hdots, d_k\) satisfies
		\(d_{i} \ge c (k-i+1)\).
        This follows by induction on \(i\).
        The initial case \(i=k\) holds because \(d_k > c \cdot \sqrt{e(G_{k-1})} \ge c \).
        For \(i \le k-1\), using the induction hypothesis,      
        \begin{eqnarray} \label{eq1}		
			d_i > c \sqrt{e(G_{i-1})}
			&\ge& c \sqrt{\sum_{j=i+1}^k d_j} \\
			&\geq& c\sqrt{\sum_{j=i+1}^kc(k-j+1)} 
			= c \sqrt{ c \binom{k-i+1}{2}}
			\ge c(k-i +1). \nonumber
        \end{eqnarray}
		
		For \(i \in [k]\), let $S_i$ be the star rooted at \(v_i\) with edges in \(G_{i-1}\), so that \(e(S_i) = d_i\).
  
		If \(e(G_k)=0\), then \(E(G) = \bigcup_i E(S_i)\) and the theorem follows from Theorem~\ref{thm:main-stars} to the sequence $d_k, \ldots, d_1$. Indeed, notice that the $i^{\text{th}}$ element in this sequence, namely $d_{k+1-i}$, satisfies $d_{k+1-i} \ge ci \ge 1600i$, so the theorem is applicable.
        Suppose \(e(G_k)\ge 1 \).
        We will build an ASD of \(G\) in the following way: 
        given an edge decomposition of \(G_k\) and a graph \(H_i\) in the decomposition, we will find a large substar 
        \(\hat{S}_i\) among \(S_1,\hdots, S_k\) that has several edges not used yet which are vertex disjoint from \(H_i\) and pick $G_i'$ of appropriate size so that \(H_i \subseteq G'_i \subseteq H_i \cup \hat{S}_i\).
        Then all the remaining edges will be substars of \(S_1,\hdots, S_k\), and they can be decomposed using Theorem~\ref{thm:main-stars}.
        
        Let \(1 \le t\le m-1\) be the integer such that
		\(
		\binom{t}{2}<e(G_k) \le \binom{t+1}{2},
		\)
        so \(\Delta(G_k) \le c\sqrt{e(G_k)}\leq c\sqrt{\binom{t+1}2}\leq c t\).
		Let $H_{r+1}, \ldots, H_{m}$ be an edge decomposition of $G_k$ as follows. If $t \ge m_0$, then by Lemma~\ref{lem:max-deg-linear} $G_k$ has an ASD consisting of $t$ graphs, in which case set $r := m-t$ and let $H_{r+1}, \ldots, H_m$ be an ASD of $G_k$. Otherwise, set $r := m-e(G_k)$, and let $H_{r+1}, \ldots, H_{m}$ be a decomposition of $G_k$ into individual edges.  Note that, by definition, each $H_j$ has at most $j-r$ edges. 
        Also observe that in both cases \(e(G\setminus G_k) \le 2mr\).
        Indeed, 
        if  $r = m - t$, we have
		\[
			e(G\setminus G_k) \le \binom{m+1}{2} - \binom{t}{2} = \binom{r+1}{2} + (r+1)(m-r) \le (r+1)m \le 2rm.
		\]
		Otherwise, \(r \ge m- \binom{m_0}{2} \ge m/2\),
        using \(m\ge m_0^2\),
        and so \(e(G\setminus G_k)\le 2mr\). 
		\begin{innerclaim}
			Let \(h=\frac{\sqrt{c}}{4\sqrt{2}}\) (so $100 \le h \le 1000$, recalling that $c = 10^6$).
			Then
			\( k \le \frac{r}{h}\).
		\end{innerclaim}
		\begin{proof}
			Let \(k'\) be maximal such that 
			\(d_{k'} \ge 4mh\).
			Then from 
			\(e(G\setminus G_k) \le 2mr\) 
			and the lower bound
			\(e(G\setminus G_k) \ge k' d_{k'} \),
			we have $k' \le \frac{r}{2h}$.
			If $t \ge \frac{m}{2}$ then $d_k \ge c \sqrt{\binom{m/2}{2}} \ge \frac{cm}{4} \ge 4mh$, so $k = k' \le \frac{r}{2h}$.

			So, we may assume that 
			$t \le \frac{m}{2}$. Then $r \ge \frac{m}{2}$, using that $r \ge m/2$ if $r \neq m-t$, and so
			\(e(G\setminus G_k) \le 2mr \le 4r^2\).
			On the other hand, by \eqref{eq1}, we have
			\(
			e(G\setminus G_k) \ge \sum_{i=k'}^k d_i \ge \sum_{i=k'}^k c\, (k-i+1) \ge \frac{c}{2}(k-k')^2 \).
			Thus
			$k - k'
			\le \sqrt{\frac{8}{c}r^2} = \frac{r}{2h}$.
		\end{proof}
			
        We define stars \(\hS_{r+1}, \hdots, \hS_m\) contained in \(S_1\cup \cdots \cup S_k\) inductively as follows.
        Let 
        \[
			T_i \in \left\{S_1\setminus 
			\left( \hS_{i+1}\cup \cdots\cup \hS_m \right),
			\hdots, S_k\setminus 
			\left( \hS_{i+1}\cup \cdots\cup \hS_m \right) \right\}
        \]
        be a star of maximum size, and let \(\hS_i\) be a substar of \(T_i \setminus V(H_i)\) of size \(i-e(H_i)\) (this is possible by the following claim).
		\begin{innerclaim}
			\(
			e(T_i \setminus V(H_i)) \ge i-e(H_i) + 20(r+1).
			\)
		\end{innerclaim}

		\begin{proof}

			Let $i \in \{r, \ldots, m\}$, and assume that the stars $\hS_{i+1}, \ldots, \hS_m$ were defined as above.
			Then
			\begin{align*}
				e(S_1 \cup \ldots \cup S_k) - e(\hS_{i+1} \cup \ldots \cup \hS_{m})
				&= \binom{m+1}{2} - e(G_k) - \sum_{j=i+1}^m (j-e(H_j)) \\
				&= \sum_{j = 1}^m j - \sum_{j = r+1}^m e(H_j) - \sum_{j = i+1}^m j + \sum_{j = i+1}^m e(H_j) \\
				&= \binom{i+1}{2} -\sum_{j=r+1}^i e(H_j) \\
				&\ge \binom{i+1}{2} -\sum_{j=r+1}^i (j-r)  \\
				& = r\left(i - \frac{r}{2} + \frac{1}{2}\right) 
				\ge r\left(i - \frac{r}{2}\right),
			\end{align*}
			where for the first inequality we used \(e(H_j) \le j-r\).
			Then
			\[
				e(T_i) \ge 
				\frac{1}{k} \cdot r\left(i - \frac{r}{2}\right)
				\ge
				h\left(i - \frac{r}{2}\right),
			\]
			using the bound \(k\le \frac{r}{h}\).
			Hence
   
			\begin{align*}
				e(T_i \setminus V(H_i)) - \big(i - e(H_i) + 20(r+1)\big)
				& \ge e(T_i) - e(H_i) - i - 20(r+1) \\
				& \ge h\left(i - \frac{r}{2}\right) - (i-r) - i - 20(r+1) \\
				& \ge (h-2)i - \frac{hr}{2} - 40r \\
				& \ge (h-2)r - \frac{hr}{2} - 40r
				= r\left(\frac{h}{2} -42\right) \ge 0.
			\end{align*}
			Here we used that the centre of $T_i$ is not in $H_i$ and $|V(H_i)| \leq 2e(H_i)$ for the first inequality, that $e(H_i) \le i-r$ for the second inequality, and that $h \ge 100$ for the last inequality.
			This proves $e(T_i \setminus V(H_i)) \ge i - e(H_i) + 20(r+1)$, as required.
		\end{proof}
		
		Define $\hG_i := \hS_i \cup H_i$ for $i \in [r+1, m]$, so \(e(\hG_i) = i\).
		Observe that the graphs 
		\(\hat{G}_{r+1}, \hdots, \hat{G}_{m} \) are pairwise edge-disjoint and cover all but \(\binom{r+1}{2}\) edges of \(G\), and all uncovered edges lie in $S_1 \cup \ldots \cup S_k$. 
        For \(i\in [k]\) let \(S_i' = S_{k-i+1} \setminus \bigcup _{j=r+1}^m \hat{G}_j\). Note that if some edges of the star $S_j$ were used by \(\hS_{r+1},\hdots, \hS_m\), then in the end of the process it still has at least $20(r+1)$ edges.
        Otherwise, by \eqref{eq1} the size of the star $S_{j}$ remains $d_j$ and is thus at least $c(k-j+1)$. Either way, $e(S_i') \ge \min\{20(r+1), ci\}$.
        Hence, by Theorem~\ref{thm:main-stars}, there is a decomposition of the last \(\binom{r+1}{2}\) edges into stars 
		\(\hG_1,\hdots, \hG_r\) 
		with \(e(\hG_i) = i\).
		
		We claim that $\hG_1, \ldots, \hG_m$ is an ASD of $G$. Indeed, first notice that $e(\hG_i) = i$. Next, we confirm that $\hG_i$ is isomorphic to a subgraph of $\hG_{i+1}$, for $i \in [m-1]$.
		This clearly holds when $i \in [r-1]$, because both $\hG_i$ and $\hG_{i+1}$ are stars.
		Next, notice that $e(H_1) = 1$, and thus $\hG_{r+1}$ is the disjoint union of a star of size $r$ and an edge, implying that $\hG_r$ is isomorphic to a subgraph of $\hG_{r+1}$.
		
		Finally, recall that $H_i$ is isomorphic to a subgraph of $H_{i+1}$ and $e(H_{i+1}) \in \{e(H_i), e(H_i)+1\}$ for $i \in [r+1,m-1]$. Since $\hS_j$ is a star of size $j-e(H_j)$, for $j \in [r+1, m]$, it follows that either $e(H_{i+1}) = e(H_i)$ and $e(\hS_{i+1}) = e(\hS_i) + 1$, or $e(H_{i+1}) = e(H_i) + 1$ and $e(\hS_{i+1}) = e(\hS_i)$. 
        Either way, $\hG_{i} = \hS_i \cup H_i$ is isomorphic to a subgraph of $\hS_{i+1} \cup H_{i+1} = \hG_{i+1}$, for $i \in [r+1,m-1]$, completing the proof that $\hG_1, \ldots, \hG_m$ is an ascending subgraph decomposition of $G$.
	\end{proof}

\section{Preliminary lemmas}

	In this section we prove various preliminary lemmas that will be used in the proof of Lemma~\ref{lem:max-deg-linear}, which shows that graphs of linear maximum degree have an ascending subgraph decomposition.

	At the end of the proof of \Cref{lem:max-deg-linear} we will need to use the Chernoff bound.
	\begin{theorem}[{Chernoff Bound, \cite[eq.\ (2.9) and Theorem 2.8]{jlr}}] \label{thm:Chernoff}
		Let \(X\) be the sum of \(n\) mutually independent indicator random variables.
		Then, for every \(t\ge 0\),
		\[
			\prob{X \ge \expect{X} + t} \le \e^{-\frac{2t^2}{n}}.
		\]
	\end{theorem}

	\subsection{Matching lemmas}

		In this section we prove some lemmas related to matchings. Several of these have appeared in the literature, but we include the proofs here for completeness.

		We say that two sets have \emph{almost equal size} if their sizes differ by at most one.
		The next lemma shows that a collection of edge-disjoint matchings can be rearranged so that the matchings have almost equal size.

		\begin{lemma}[cf.\ \cite{deWerra}, \cite{Mcdiarmid}] \label{lem:balancing-matching}
			Let $M_1, \ldots, M_k$ be a collection of pairwise edge-disjoint matchings in a graph $G$. Then there is a collection $M_1', \ldots, M_k'$ of pairwise edge-disjoint matchings of almost equal size such that $\bigcup M_i = \bigcup M_i'$.
		\end{lemma}
		\begin{proof}
            Without loss of generality
            \(\abs{M_1} \le \hdots \le \abs{M_k}\)
            and suppose \(\abs{M_k} \ge \abs{M_1} +2\) since otherwise we are done.

            The connected components of $M_1\cup M_k$ are paths and even cycles, such that in both cases consecutive edges belong to different matchings. Notice that in even paths and cycles the number of edges belonging to each of the two matchings is the same. Hence, there is a connected component \(P\) that is an odd path whose first and last edges lie on $M_k$.
            Swap the edges of \(M_1,M_k\) on \(P\) and let \(M_1', M_k'\) be the resulting matchings.
            Then \(\abs{M_1'} = \abs{M_1}+1\), \(\abs{M_k'} = \abs{M_k} -1\) and \(M_1', M_k'\) are edge-disjoint.

            If the matchings \(M_1', M_2, \hdots, M_{k-1}, M_k'\) are not almost equal, iterating the above argument either decreases the maximum difference in size between two matching, or decreases the number of 
            pairs of matchings whose difference in size is at least 2 and maximum.
            This procedure must eventually terminate, yielding a collection of almost equal matchings decomposing $M_1 \cup \ldots \cup M_k$.
		\end{proof}

		Next we use Vizing's theorem together with Lemma \ref{lem:balancing-matching} to prove that graphs with small maximum degree have ascending subgraph decompositions into matchings; this fact was (essentially) already proved by Fu~\cite{fu1990note}.
		\begin{lemma}[cf.\ \cite{fu1990note}] \label{lem:asd-bounded-degree}
			Let $G$ be a graph with $e$ edges, where $\binom{m}{2} < e \le \binom{m+1}{2}$, and suppose that $G$ has maximum degree at most $\floor{m/2}-1$. Then $G$ has an ascending subgraph decomposition into matchings.
		\end{lemma}
		\begin{proof}
			Let \( t = e - \binom{m}{2}\le m\).
			By Vizing's theorem we can decompose \(E(G)\) into at most \(\floor{m/2}\) matchings. Thus, as the following calculation shows, there exists a matching \(M^\ast\) of size \(t\):
			\(
				\frac{e}{\floor{m/2}} 
				> \frac{\binom{m}{2}}{m/2}
				= m-1.
			\)
			Let $G' = G - M^*$, so that
			\(e(G') =  \binom{m}{2}\).
			We now consider two cases based on the parity of \(m\).

			If \(m\) is even, by Vizing's theorem and Lemma~\ref{lem:balancing-matching}, we can decomposes $G'$ into matchings \(M_1,\hdots, M_{m/2}\) of size \(m-1\) each.
			Let \(H_{m-1}:= M_{m/2}\).
			For \(i\in [m/2-1]\) let \(H_i\) consist of \(i\) edges from \(M_i\) and let \(H_{m-i-1} := M_i - H_i\).
			Then 
            \((H_1, \hdots, H_{t-1}, M^\ast, H_{t},\hdots, H_{m-1})\) is an ascending subgraph decomposition.
			
			If \(m\) is odd, by Vizing and Lemma~\ref{lem:balancing-matching}, there is a decomposition of $G'$ into matchings \(M_1,\hdots, M_{(m-1)/2}\) of size \(m\) each.
			For \(i\in [(m-1)/2]\) let \(H_i\) consist of \(i\) edges of \(M_i\), and let \(H_{m-i}:= M_i - H_i\).
			Then \((H_1, \hdots, H_{t-1}, M^\ast, H_t, \hdots, H_{m-1})\) is an ascending subgraph decomposition of \(G\).
		\end{proof}

		The next lemma is a generalisation of Hall's theorem, that in our  context yields a decomposition of a bipartite graph into isomorphic star forests and a graph of small maximum degree.

		\begin{lemma}
			\label{lem:star-forests}
			Let $H$ be a bipartite graph with bipartition $\{X, Y\}$ such that $d(x) < d$ for every $x \in X$. Then $H$ can be decomposed into $(\SF_1, \ldots, \SF_d, R)$, where the $\SF_i$'s are isomorphic star forests of at most 
            \(\abs{Y}\) components whose stars have size at most $\Delta(H)/d$, and $\Delta(R) < d$.
		\end{lemma}
		\begin{proof}
			Let \(Y_1\) be the set of vertices in \(Y\) of degree at least \(d\). Let $R$ be a subgraph of $H$ consisting of $d(y)\pmod d\leq d-1$ edges through each $y\in Y$, noting that it contains all edges through $Y\setminus Y_1$ and has $\Delta(R)<d$. Moreover, the degrees of all the vertices from $Y_1$ in graph $H\setminus R$ are divisible by $d$.
			We define a new bipartite graph $H'$ as follows.
			For each \(y\in Y_1\), introduce vertices \(y_1,\hdots, y_{\floor{d(y)/d}}\) and set $Y'=\bigcup_{y\in Y_1}\{y_1,\hdots, y_{\floor{d(y)/d}}\}$. Let the two parts of $H'$ be $X$ and $Y'$.
			For the edges of $H'$, split $N_{H\setminus R}(y)$ into disjoint sets $N_1, \ldots, N_{d(y)/d}$ of size $d$, and join each $y_i$ to all the vertices in $N_i$. This way, contracting each set $\{y_1,\hdots, y_{\floor{d(y)/d}}\}$ into a single vertex turns $H'$ into $H\setminus R$.

			Notice that $H'$ is a bipartite graph with maximum degree at most $d$. Then by K\"onig's theorem, \(E(H\prm)\) can be decomposed into \(d\) matchings $M_1, \ldots, M_d$. 
			Since all vertices in \(Y\prm\) have degree exactly \(d\), each matching uses precisely one edge incident to each \(y_i\).
			Thus, each matching $M_i$ corresponds to a star forest $\SF_i$ in $H$ whose stars are centred at $Y_1$ and such that the star centred at $y$ has size \(\floor{d(y)/d}\), for $y \in Y_1$. 
			In particular, the star forests $\SF_1, \ldots, \SF_d$ are isomorphic, have \(\abs{Y_1}\) components, and their stars have size at most $\Delta(H)/d$.
		\end{proof}

	\subsection{Combining graphs}

		A graph $G$ is $r$-divisible, if for every graph $H$, the number of connected components of $G$ which are isomorphic to $H$ is divisible by $r$. Note that every $r$-divisible graph $G$ can be edge-decomposed into $r$ isomorphic vertex-disjoint subgraphs, consisting of a $1/r$ fraction  of components belonging in the same isomorphism class. Use $G/r$ to denote the isomorphism class of these graphs. 
		In this section we prove two lemmas that will be used to combine divisible graphs with stars and matchings.

		The following lemma is used for combining divisible graphs with stars.
		\begin{lemma}\label{lem:combine-star-with-forest}
			Let $H$ be $4$-divisible, $S$ a star of even size, and $G$ have a decomposition $(H, S)$. Then there is a decomposition $(H_1, \ldots, H_4, S_1, \ldots, S_4)$ of $G$ such that: $H_i, S_i$ are vertex-disjoint for $i \in [4]$; each of the graphs $H_1, \dots, H_4$ is isomorphic to $H/4$; and $S_1, \ldots, S_4$ are stars with $e(S_1)=e(S)/2$.
		\end{lemma}
		\begin{proof}
			Let $(H_1, \ldots, H_4)$ be a decomposition of $H$ into four copies of $H/4$.
			Without loss of generality, the centre $c$ of $S$ is not in $H_1\cup H_2 \cup H_3$. Order $H_1, H_2, H_3$ so that $|V(S)\cap V(H_1)| \leq |V(S) \cap V(H_2)| \leq |V(S)\cap V(H_3)|$, noting that this gives $|V(S)\cap V(H_1)|\leq e(S)/3$. 
			Then $|V(S)\setminus V(H_1)|\geq 2e(S)/3$, so we can pick a star $S_1$ of size $e(S)/2$ disjoint from $H_1$. Let $S_2$ be the subgraph of $S \setminus S_1$ whose edges touch $H_3$, let $S_3=S\setminus (S_1\cup S_2)$, and set $S_4=\emptyset$.
            Then  $H_i, S_i$ are vertex-disjoint for all $i$: for $i \in \{1, 3, 4\}$ this is trivial by construction; and for $i = 2$ this happens because $H_2, H_3$ are vertex-disjoint and $V(S_2)\subseteq V(H_3)\cup\{c\}$. 
		\end{proof}

		For graphs $H_1, H_2$ and integer $a$, the sum $H_1 + H_2$ is the disjoint union of $H_1$ and $H_2$ and $a \cdot H_1$ is the disjoint union of $a$ copies of $H_1$. The next lemma combines divisible graphs with matchings.

		\begin{lemma} \label{lem:large-matching-forest}
			Let \(\ell, a_1,\hdots, a_k\) be positive integers.
			Let $F_1, \ldots, F_k, H$ be graphs satisfying $H \cong 5a_1 \cdot F_1 + \ldots + 5a_k \cdot F_k$, 
			and let $M$ be a matching of size $5\ell$.
			If
			\begin{equation} \label{eq:large-matching-forest}
				\ell > \sqrt{ \frac{\ln5}{2}\, \sum_{j=1}^k a_j \abs{V(F_j)}^2},
			\end{equation}
			then there is a decomposition of $M \cup H$ into five graphs, each of which is isomorphic to $H/5 + M/5$ for $i \in [5]$.
		\end{lemma}

		We now state McDiarmid's inequality, which will be used in the proof of the last lemma.

		\begin{theorem}[McDiarmid's inequality] \label{thm:mcdiarmid}
			Let \(X_1,\hdots,X_m\) be independent random variables with \(X_i\) taking values in a set \(S_i\).
			Let \(f: \prod_{i\in [m]} S_i \rightarrow \R\) be a function such that
			for any \(\xb, \xb\prm \in \prod_{i\in [m]} S_i\) differing only at the \(k^{\text{th}}\) coordinate we have
			\[
				\abs{f(\xb) - f(\xb\prm)} \le c_k,
			\]
			for some \(c_k \in \R\).
			Then, for every $t > 0$,
			\[
			\prob{f(X_1,\hdots,X_m)\le \expect{f(X_1,\hdots,X_m)} - t}
			\le \exp\left( -\frac{2t^2}{\sum_{k=1}^m c_k^2} \right).
			\]
		\end{theorem}

		\begin{proof}[Proof of Lemma~\ref{lem:large-matching-forest}]
			For each \(j\in[k]\) partition the components of $H$ isomorphic to \(F_j\) into \(a_j\) sets of size 5 each.
			Permute uniformly at random the five copies of \(F_j\) in each set and 
            for \(j \in [k]\), \(s\in [a_j]\) let
			\(
			 X_{j,s} 
			\)
			be the resulting random permutation for the $s^{\text{th}}$ set of copies of $F_j$.
			This defines a random partition of \(H\) into five graphs \(H_1, \hdots, H_5\), where \(H_i\) consists of the copies of \(F_1,\hdots,F_k\) at the \(i^{\text{th}}\) position in each set, and so $H_i \cong H/5$.
			We will show that, with positive probability, every \(H_i\) is vertex-disjoint of more than \(2\ell\) edges of \(M\).
			
			For \(e\in M,\, i\in [5]\), let \(Y_i^e\) be the indicator random  variable for the event that \(e\) does not share a vertex with \(H_i\).
			Then
			\(
			\expect{Y_i^e} \geq 3/5,
			\)
			since the endpoints of \(e\) can lie on at most two of the graphs \(H_1,\hdots, H_5\).
			Let
			\(
			Y_i = \sum_{e\in M} Y_i^e
			\)
			be the number of edges in \(M\) that \(H_i\) is vertex-disjoint of.
			Then
			\(
			\expect{Y_i} \geq \frac{3}{5} \abs{M} = 3 \ell.
			\)
			Observe that the random variables \(Y_i^e\) and hence \(Y_i\) are determined by 
			\(
			( X_{j,s} )_{j\in[k], s\in[a_j]}.
			\)
			Changing the value of \(X_{j,s}\) for one pair $(j, s)$ changes the value of \(Y_i\) by at most \(\abs{V(F_j)}\).
			Hence, by Theorem~\ref{thm:mcdiarmid} (McDiarmid's inequality),
			\[
				\prob{Y_i \leq 2l} \leq
				\prob{Y_i \le \expect{Y_i} - \ell} \le
				\exp\left( -
				\frac{2\ell^2}{\sum_{j=1}^k a_j \abs{V(F_j)}^2}
				\right),
			\]
			which is less than \(1/5\) for
			\(
			\ell > \sqrt{ \frac{\ln 5}{2}\, \sum_j a_j \abs{V(F_j)}^2}.
			\)
			By taking a union bound over \(i\in [5]\) we deduce that there exists a decomposition of \(H\) into \(H_1,\ldots, H_5\) such that each \(H_i\) is isomorphic to $H/5$ and is vertex-disjoint of more than \(2\ell\) edges of \(M\).
			
			Consider the bipartite graph \(\calG\) with bipartition \((M,\calH)\), where 
            \(\calH\) consists of \(\ell\) copies of each \(H_i\), so $\abs{M} = \abs{\calH}=5\ell$.
			For \(e\in M,\, H_i\in \calH\), put \(e H_i\in E(\calG)\) if and only if the edge \(e\) is vertex-disjoint of the graph \(H_i\).
			Then we have \(d_\calG(e) \geq 3\ell\) and \(d_\calG(H_i) > 2\ell\) for all \(e\in M, H_i\in \calH\). 
			It is easy to see that \(\calG\) satisfies Hall's condition:
			if \(S\subseteq M\) has \(\abs{S} \leq 3\ell,\) clearly \(\abs{N(S)} \geq 3\ell \geq \abs{S}\); 
			and if \(\abs{S}>3\ell\), then 
            \(\abs{S}>\abs{M} - d_\calG(H_i)\) for all \(H_i \in \calH\), so
            \(N(S) = \calH\).
			Hence \(\calG\) has a perfect matching.
			This gives a decomposition of \(M\cup H\) as follows: let \(G_i\) be the union of \(H_i\) and the edges of \(M\) that are matched with the copies of \(H_i\) in \(\calG\).
			Then \(G_i \cong H/5 + M/5\), and so $(G_1, \hdots, G_5)$ satisfies the requirements of the lemma.
		\end{proof}

	\subsection{Almost decomposing dense graphs into $K_{t,t}$-forests}
		
		An \emph{$F$-forest} is a disjoint union of copies of $F$.
		The aim of this subsection is to prove the following lemma, which almost decomposes a dense graph into isomorphic $K_{t,t}$-forests.
		
		\begin{lemma} \label{lem:Ktt-decomposition}
			Let $t \ge 1$ be a fixed integer, let $n$ be a large integer, and let $k$ be an integer satisfying 
           \(\frac{n}{\sqrt{t}} \le k \le \frac{n^2}{t^{5/2}}
            \). 
			Suppose that $G$ is a graph on $n$ vertices. Then $G$ can be decomposed into \(k\) isomorphic $K_{t,t}$-forests and a remainder of at most $4n^2/\sqrt{t}$ edges. 
		\end{lemma}

		The proof will use the following result due to Pippenger and Spencer \cite{pippenger-spencer}. 
		For a hypergraph \(\calH\) and \(v\in V(\calH)\), the \emph{degree} of \(v\), denoted \(d(v)\), is the number of edges incident to \(v\).
		We denote the maximum degree of \(\calH\) by \(\Delta(\calH)\) and the minimum degree by \(\delta(\calH)\).
		For \(u\neq v\), the \emph{codegree} of \(u,v\), denoted by \(d(u,v)\), is the number of edges incident to both \(u\) and \(v\).
		A \emph{matching} in a hypergraph is a collection of pairwise vertex-disjoint edges.
        
		\begin{theorem}[Pippenger--Spencer,~{\cite[Theorem 1.1]{pippenger-spencer}}]\label{thm:pipp-spencer}
			Let $r$ be a positive integer and \(\mu>0\).
            Then for sufficiently small \(\nu>0\) and for sufficiently large $n$ the following holds.
		If $\calH$ is an $r$-uniform hypergraph on \(n\) vertices with 
            \(\delta(\calH) \ge (1-\nu) \Delta(\calH) \)
            and
            $d(u,v) \le \nu \Delta(\calH)$, for all distinct \(u, v \in V(\calH)\), then
	there is a decomposition of $\calH$ into at most \((1+\mu)\Delta(\calH)\) matchings.
		\end{theorem}

		In fact, we shall need the following easy corollary.

		\begin{corollary} \label{cor:pipp-spencer}
			Let $r$ be a positive integer
            and \(\mu>0\).
            Then for sufficiently small $\nu>0$ and for sufficiently large $n$ the following holds.
            If $\calH$ is an $r$-uniform hypergraph on \(n\) vertices with $d(u,v) \le \nu \Delta(\calH)$, for all \(u, v \in V(\calH)\),
            then there is a decomposition of $\calH$ into at most $(1 + \mu)\Delta(\calH)$ matchings.
		\end{corollary}

		\begin{proof}
			Write $d = \Delta(\calH)$.
			We will embed \(\calH\) in a $d$-regular hypergraph \(\calH\prm\) with the same maximum codegree as \(\calH\).
			Take \(r\) copies of \(\calH\) and add an edge through all copies of \(v\in V(\calH\)) for any \(v\) with \(d(v) < d\).
			This increases the minimum degree of the hypergraph by 1, does not increase the codegree (pairs of copies of \(v\) have codegree 1) and does not increase the maximum degree. 
			By repeating this construction at most \(d-1\) times we obtain a hypergraph \(\calH\prm\) as desired.
			
			Apply Theorem~\ref{thm:pipp-spencer} to $\calH\prm$ to decompose it into at most $(1 + \mu)d$ matchings. This induces a decomposition of $\calH$ into at most $(1 + \mu)d$ matchings, as required.
		\end{proof}

		Our proof will also use the following simple version of the classical K\H{o}v\'ari--S\'os--Tur\'an theorem.

		\begin{theorem}[K\H{o}v\'ari--S\'os--Tur\'an] \label{thm:kst}
			Let \(t\) be a positive integer and \(n\) be sufficiently large.
            Let \(G\) be a graph on \(n\) vertices with no subgraph isomorphic to \(K_{t,t}\).
			Then \(G\) has at most
			\( n^{2-1/t}\)
			edges.
		\end{theorem}

		Finally, recall that a \emph{projective plane of order \(p\)} is a \((p+1)\)-uniform hypergraph on \(p^2 + p +1\) vertices, also called \emph{points}, with \(p^2 + p +1\) edges, also called \emph{lines}, such that any two lines meet at exactly one point and any two points lie in exactly one common line.
		There exist constructions of projective planes of order $p$ for any prime (see e.g.~\cite[section 12.4]{jukna}).
		
		The main step in the proof of Lemma~\ref{lem:Ktt-decomposition} is the following lemma, which almost decomposes a dense graph into a small number of $K_{t,t}$-forests (whose sizes may differ).
		 
		\begin{lemma}[decomposing into $K_{t,t}$-forests] \label{lem:unequal-Ktt-decomposition}
			Fix \(\alpha,t>0\) and let \(n\) be large.
			Suppose $G$ is a graph on \(n\) vertices with 
            \(e(G) \ge \alpha n^2 \).
            Then $G$ can be decomposed into at most $2n/t$ many $K_{t,t}$-forests and a remainder of at most $ 45n^{2-1/2t}$ edges.
		\end{lemma}
		\begin{proof}
			Use Bertrand's postulate to pick a prime $p\in [\sqrt{n}, 2\sqrt{n} ]$, and let \(\calP\) be a projective plane of order \(p\) with \(q=p^2+p+1\) points and lines.
			Let \(f\) be an injection from \(V(G)\) to the points of \(\calP\), and for \(j\in [q]\) let \(H_j\) be the subgraph of \(G\) consisting of edges \(uv\) such that \(f(u), f(v)\) both lie in the \(j\)-th line of \(\calP\).
			Observe that \(H_1, \hdots, H_q\) decompose \(G\), since for any \(uv\in E(G)\), a unique line of $\calP$ goes through \(f(u),f(v)\).
            Moreover since each line of \(\calP\) has size \(p+1\), \(\abs{H_j} \le p+1 \le 3\sqrt{n}\).
			
			By the K\H{o}v\'ari--S\'os--Tur\'an theorem (Theorem~\ref{thm:kst}), we can greedily remove copies of \(K_{t,t}\)
            from each graph \(H_j\), until there are at most
			\(
			\abs{H_j}^{2-1/t}
			\)
			edges left. 
			Let $\calk$ be the collection of $K_{t,t}$-copies removed in this process.
			In total, the $K_{t,t}$-copies in $\calK$ cover all but at most
            \[
			\sum_{j=1}^{q}  \abs{H_j}^{2-1/t} \le 45 n^{2-1/2t},
			\]   
			edges of \(G\), using that $q \le 4n + 2\sqrt{n} + 1 \le 5n$.
            
			Let \(\calH\) be the \(2t\)-uniform hypergraph on \(V(G)\) 
            where a set $U$ of $2t$ vertices is an edge if and only if it is the vertex set of a \(K_{t,t}\) copy in $\calk$.
            Notice that the edges incident to a vertex \(v\) split into at most \(n/t\) stars of size \(t\), so each vertex is covered by at most \(n/t\) copies of
            \(K_{t,t}\), i.e.\ $\Delta(\calH) \le n/t$.
            Moreover, since the number of edges covered by $K_{t,t}$-copies in $\calk$ is at least
            \(e(G) -  45n^{2-1/2t}
            \ge \alpha n^2 /2\),
            some vertex $v$ has at least \(\alpha n\) incident edges in the decomposition, and hence, since each \(K_{t,t}\) covers \(t\) edges incident to $v$,
            $v$ lies in at least
            \(\frac{\alpha n}{t}\)
            copies of \(K_{t,t}\) in $\calk$, i.e.\ 
            \(\Delta(\calH) \ge \frac{\alpha n}{t}\).

			For any two vertices \(u,v\), the only copies of \(K_{t,t}\) in $\calk$ that can cover both $u$ and $v$ lie in the graph \(H_j\) corresponding to the unique line through \(f(u), f(v)\).
            The number of \(K_{t,t}\)-copies in $\calk$ that lie in \(H_j\) and contain \(u,v\) is 
            (crudely) at most the number of edges incident to \(u\) in \(H_j\), which is less than \(\abs{H_j}\). 
			Thus, there are at most 
            \(
            \abs{H_j} \le p+1 \le 3 \sqrt{n}
            \)
            copies of \(K_{t,t}\) that contain both \(u,v\).
			It follows that the codegree of \(\calH\) is at most
			\(3 \sqrt{n} = o(\Delta(\calH))\).
			Hence, by Corollary~\ref{cor:pipp-spencer}, we can decompose \(\calH\) into at most 
           $(1+\mu)\Delta(\calH) \leq 2 n/t$  matchings.
			This corresponds to a partition of the above approximate \(K_{t,t}\)-decomposition of $G$ into at most \(2 n/t\) many \(K_{t,t}\)-forests.
		\end{proof}
		
		The next proposition tells us that given $\ell$ parts of arbitrary size, we can further partition each part, to obtain a total of $k$ parts of equal size and a small number of unused elements, provided that $k$ is sufficiently larger than $\ell$. We will apply it to prove Lemma~\ref{lem:Ktt-decomposition}.
		
		\begin{proposition} \label{prop:division}
			Let $\ell, k, s_1, \ldots, s_{\ell}$ be positive integers satisfying $\sum_i s_i \ge k+\ell$. Then there exists a positive integer \(s\) and non-negative integers $\sigma_1, \ldots, 
			\sigma_{\ell}$ such that $\sigma_i s \le s_i$ for $i \in [\ell]$, $\sum_i \sigma_i = k$ and $\sum_i s_i - sk \le \frac{\ell}{k}\sum_i s_i + k$. 
		\end{proposition}
		\begin{proof}
			Let $s := \floor{\sum_i{s_i}/(k+\ell)}$, let $s_i' \in [s_i - s + 1, s_i]$ be divisible by $s$ (notice that there is a unique such $s_i'$) for $i \in [\ell]$. Write $\sigma_i' = s_i'/s$ for $i \in [\ell]$. 
			
			Clearly, $s \le \sum_i s_i / (k + \ell)$. Equivalently,		
				$sk \le \sum_{i = 1}^\ell s_i - s\ell = \sum_i (s_i - s) < \sum_i s_i'$.
			By dividing both sides by $s$, it follows that $\sum_i \sigma_i' \ge k$, which implies that we can pick $\sigma_i \in [0, \sigma_i']$ such that $\sum_i \sigma_i = k$. Note that $\sigma_i s\leq \sigma_i's=s_i'\leq s_i$ as required.
			As $s \ge \sum_i s_i / (k + \ell) - 1$,
			\begin{equation*} 
				\sum_i s_i - ks 
				\le \sum_i s_i - k \left(\frac{\sum_i s_i}{k + \ell} - 1\right)
				= \frac{\ell}{k + \ell}\sum_i s_i + k
				\le \frac{\ell}{k}\sum_i s_i + k.
			\end{equation*}
			Thus $\sigma_1, \ldots, \sigma_{\ell}, s$ satisfy the requirements of the proposition.
		\end{proof}

		We now prove the main lemma of this subsection.
		\begin{proof}[Proof of Lemma~\ref{lem:Ktt-decomposition}]
            Suppose \(e(G) \ge \frac{4n^2}{\sqrt{t}}\), since otherwise we can fit all edges of \(G\) in the remainder.
			Apply Lemma~\ref{lem:unequal-Ktt-decomposition} to $G$, to obtain a decomposition of $G$ into $K_{t,t}$-forests $\KF_1, \ldots, \KF_{\ell}$, with $\ell \le 2n/t$, and a remainder of at most $ 45n^{2-1/2t}$ edges.
			Let $s_i$ be the number of $K_{t,t}$-copies in $\KF_i$. 
            Then 
            $\sum_i s_i$ is the number of edges of $G$ covered by $\bigcup_i \KF_i$, divided by $t^2$.
			Thus
            \[
            \sum_{i=1}^{\ell} s_i 
            \ge \frac{e(G) - 45n^{2-1/2t}}{t^2}
            \ge \frac{e(G)}{2t^2}
            \ge \frac{2n^2}{t^2 \sqrt{t}}
            \ge \frac{n^2}{t^{5/2}} + \frac{2n}{t}
            \ge k + \ell.
            \]
            Therefore, by Proposition~\ref{prop:division}, there exist non-negative integers $\sigma_1, \ldots, \sigma_{\ell}, s$ such that $\sigma_i s \le s_i$ for $i \in [\ell]$, $\sum_i{\sigma_i} = k$ and 
			\begin{equation} \label{eqn:remainder}
                \sum_i (s_i - \sigma_i s)
                =
				\sum_i s_i - sk 
				\le \frac{\ell}{k}\sum_i s_i + k 
				\le \frac{2n/t}{n/\sqrt{t}} \cdot \frac{n^2}{t^2} + \frac{n^2}{t^{5/2}}
				= \frac{3n^2}{t^{5/2}},
			\end{equation}
			where the second inequality follows from the bounds on $\ell, k$ and the fact that \(\sum_i s_i \le n^2/t^2\).
			
			Let $\KF_{i,j}$, with $i \in [\ell]$ and 
            $j \in [\sigma_i]$, 
            consist of $s$ different copies of $K_{t,t}$ from $\KF_i$, such that the $\KF_{i,j}$'s are pairwise vertex-disjoint (this is possible because $\KF_i$ consists of at least $\sigma_i s$ copies of $K_{t,t}$). 
            Since $\sum_i \sigma_i = k$, the collection $(\KF_{i,j})$ consists of exactly $k$ many $K_{t,t}$-forests, which all have the same size. 
            By \eqref{eqn:remainder} and the fact that
            $\KF_1, \ldots, \KF_{\ell}$ cover all but at most \(45 n^{2-1/2t}\)  edges of \(G\),
            the number of edges in $G$ uncovered by $(\KF_{i,j})$ is at most
			\begin{equation*}
				\frac{3n^2}{\sqrt{t}} +  45n^{2-1/2t}  \le
				\frac{4n^2}{\sqrt{t}},
			\end{equation*}
			using that $n$ is large.
			The collection $(\KF_{i,j})$ satisfies the requirements of the lemma.
		\end{proof}

\section{Ascending subgraph decompositions for graphs with linear maximum degree}  \label{sec:bounded-degree}

In this section we prove Lemma~\ref{lem:max-deg-linear}, asserting that graphs with linear maximum degree have ascending subgraph decompositions.

In \Cref{lem:approx-star-forest-decomposition,lem:approx-isomorphic-decomposition,lem:approx-isomorphic-decomposition-stronger}
we will successively get finer edge decompositions of a graph with linear maximum degree.

\begin{lemma}\label{lem:approx-star-forest-decomposition}
    Let \(\eps\in(0,\frac1{10})\) be fixed, set \(c:=10^6\), and let \(m\) be sufficiently large.
Suppose that $G$ is a graph with $\Delta(G) \le cm$ and $e(G)\leq m^2$. Then for every $k \in[\eps m, m]$ there is a decomposition $(H_1, \ldots, H_k, R_1, R_2)$ of $G$ such that  $H_i$ are isomorphic star forests, having components of size at most 
$s= \ceil{5\, \eps^{-1} c }$,  $e(R_1)\leq \eps m^2$ and $|V(R_2)|\leq sm$.
\end{lemma}

\begin{proof} 
	Take $k' := \ceil{k/5}$.
	Let $L$ be the set of vertices in $G$ with degree at least $k'$, and let $S := V(G) - L$.
    Then \(m^2 \ge e(G)\geq \frac12\sum_{v\in L}d(v) \ge k' \abs{L}/2\), so 
    \(\abs{L} \le 2m^2/k' \le 10 m \eps^{-1} \le s m\).

    Apply Lemma~\ref{lem:star-forests} (with $X=S, Y=L, d=k'$) to obtain a decomposition
    $(\SF_1, \ldots, \SF_{k'}, R_0)$ of 
    \(G[S, L]\)
    where the $\SF_i$'s are isomorphic star forests with at most \(\abs{L} \le sm\) components, stars of size at most 
    $cm/k' \le 5 c \eps^{-1} \le s$, 
	and $R_0$ has maximum degree less than $k'$.
    Set \(R:= R_0 \cup G[S]\),
    and observe that the maximum degree of \(R\) is still less than \(k'\) (i.e.\ $\Delta(R)\leq k'-1$), since adding the edges of \(G[S]\) can only affect the degree of vertices in \(S\), which even in $G$ have degree at most $k'-1$.
    Let $\SF$ be the isomorphism class of the $\SF_i$'s. We assume that $\SF$ is $5$-divisible, by removing up to four components of each size from each $\SF_i$ (these components will be placed in $R_1$ at the end of the proof). 

	If $e(R) \le \eps^2 m^2$, split each $\SF_i$ into five copies of $\SF/5$, denoting the collection of copies thus obtained by $H_1, \ldots, H_{5k'}$. 
	Otherwise, using \Cref{lem:large-matching-forest}, we will define $H_1, \ldots, H_{5k'}$ to be a collection of graphs consisting of the union of \(\SF_i\) and a matching in \(R\).
    We construct this collection as follows.
    Apply Vizing's theorem and Lemma~\ref{lem:balancing-matching} to decompose $R$ into almost equal matchings $M_1, \ldots, M_{k'}$. By removing up to five edges from each $M_i$, we may assume that the $M_i$'s have the same size and are $5$-divisible. Denote the isomorphism class of the $M_i$'s by $M$, and notice that $e(M) \ge \eps^2m^2 / k' - 5 \ge m^{2/3}$. 
    Next, in order to apply \Cref{lem:large-matching-forest} to obtain a decomposition of each \( M_i\cup \SF_i \), we verify~\eqref{eq:large-matching-forest}. Recalling that the number of components in \(\SF\) is at most \(sm\) and each has size at most \(s\), we see
    the right hand side of~\eqref{eq:large-matching-forest} 
    is upper bounded by
    \(
    \sqrt{s^3 m} = o(m^{2/3}).
    \)
    Hence \Cref{lem:large-matching-forest} implies that each graph $M_i\cup \SF_i$ can be decomposed into five copies of $M / 5 + \SF / 5$. 
    Using this decomposition for each \(M_i \cup \SF_i, i\in [k']\),
    denote the collection of copies of \(M / 5 + \SF / 5\) obtained by $H_1, \ldots, H_{5k'}$. 

	Let $R_2 := G[L]$, so \(\abs{V(R_2)} \le sm\) by the calculation at the beginning of the proof.
	Let $R_1$ be the subgraph of $G - G[L]$ spanned by the edges uncovered by 
	$H_1, \ldots, H_k$.
	Then $R_1$ consists of the (at most $5k'-k \leq 4$) graphs $H_{k+1}, \ldots, H_{5k'}$, which in total cover at most 
	$4m^2/k \le 4m/\eps$ edges;
    any edges we removed from \(\cup_i \SF_i\) and \(\cup_i M_i\) for divisibility purposes; and all edges of $R$ if $e(R) \le \eps^2m^2$.
    For the second contribution, since the components of \(\SF\) are stars of size at most \(s\),
    there are at most $s$ potential component sizes, each contributing at most $s$ deleted edges. Hence we removed at most \(4s^2\) edges from each \(SF_i\), so the total contribution to \(R_1\) is \(4 s^2 k'\); and we removed at most \(5k'\) edges
    from \(\cup_i M_i\).
    Altogether, 
    $e(R_1) \le \eps m^2$.
	Therefore 
	$(H_1, \ldots, H_k, R_1, R_2)$ is a decomposition of $G$ satisfying the required properties.
\end{proof}

Next, we almost decompose the part with linearly many vertices above (namely $R_2$ in \Cref{lem:approx-star-forest-decomposition}) and combine them with the isomorphic graphs obtained in the previous step, to decompose the whole graph into isomorphic graphs with some nice properties and a remainder with few edges.

\begin{lemma}\label{lem:approx-isomorphic-decomposition}
    Fix constants \(\eps\in (0,\frac{1}{10})\) and  \(r\in \mathbb N\), set
    \(c:=10^6\), 
    \(t:=\ceil{\eps^{-12} c^6}^2\) and suppose \(m\) is sufficiently large.
	Let $G$ be a graph satisfying 
    $\Delta(G) \leq cm$ and $e(G)\leq m^2$. 
    Then for any  $k \in [\eps m, m]$ there is a decomposition 
    $(H_1, \ldots, H_{k}, R)$ 
    of \(G\) such that: the $H_i$'s are isomorphic, $r$-divisible, and their components have size at most $t$; and $e(R)\leq \eps m^2$.
\end{lemma}
\begin{proof}
Let $t' := \sqrt{t}$ and \(s:= \ceil{5\, \eps^{-2} c }\), and notice that we have $t'\geq \eps^{-12}c^6\geq cs^4\eps^{-4}$. Moreover as $r, c, t, \eps$ are fixed and $m$ is large enough, $m$ is larger than any fixed function of $r, c, t, \eps$.

Apply Lemma~\ref{lem:approx-star-forest-decomposition} 
    (with $\eps_{\ref{lem:approx-star-forest-decomposition}} = \eps^2$) to obtain a decomposition 
    $(\SF_1, \ldots, \SF_k, R_1, R_2)$, 
    where the $\SF_i$'s are isomorphic star forests whose stars have size at most $s$;
    $e(R_1) \le \eps^2 m^2$; 
    and $\abs{V(R_2)} \le sm$. Let $\SF$ be the isomorphism class of the $\SF_i$'s.
    
    If 
    \(
    \abs{V(R_2)} \le t^{3/4}\, \sqrt{m}
    \), 
    then 
    \(
    e(R_2) \le \abs{V(R_2)}^2 
    \le t^{3/2}m
    \le \eps^2 m^2.
    \)
    Let \(H\) be the subgraph of \(\SF\)
    that results from  removing at most \(r-1\) stars of each size from \(\SF\), so that \(H\) is \(r\)-divisible.
    Let \(H_i\) be a subgraph of \(\SF_i\) with
    \(H_i \cong H\), and 
    \(
    R_3 = \bigcup_i \SF_i \setminus H_i
    \).
    Since there are at most \(s\) sizes of stars in \(\SF\) and the largest star has size $s$,
    \(e(R_3) \le k s^2 r\leq ms^2r\leq \eps^2 m^2\).
    Then \(R = R_1 \cup R_2 \cup R_3\) consists of the edges of \(G\) uncovered by \(H_1,\hdots, H_k\), and has size at most
    \(
    3 \eps^2 m^2  \le \eps m^2.
    \)
    Hence, since \(H\) consists of stars of size at most \(s\le t\), the decomposition 
    \((H_1,\hdots, H_k,R)\) 
    satisfies the requirements of the lemma.
    
    Thus for the remainder of the proof we  assume
    \(
    \abs{V(R_2)} \ge t^{3/4} \sqrt{m}.
    \)
    In this case we will first apply \Cref{lem:Ktt-decomposition} to decompose \(R_2\) into \(K_{t',t'}\)-forests.
    Then, by removing a few components from this decomposition and from each \(\SF_i\), we can combine each \(K_{t',t'}\)-forest with a \(\SF_i\) to obtain a decomposition as in the lemma.

	We apply \Cref{lem:Ktt-decomposition} (with $G_{\ref{lem:Ktt-decomposition}}=R_2$, $t_{\ref{lem:Ktt-decomposition}}=t'$, $k_{\ref{lem:Ktt-decomposition}}=k$, $n_{\ref{lem:Ktt-decomposition}}=|V(R_2)|\in [t^{3/4}\sqrt m, sm]$) to decompose \(R_2\) into 
    \(k\) \(K_{t',t'}\)-forests. This can be done because our $k$ satisfies the conditions of the lemma. Indeed, 
    the lower bound on 
    \(\abs{V(R_2)} \) implies
    \(
    \frac{\abs{V(R_2)}^2}{(t')^{5/2}}
	\ge \frac{t^{3/2}m}{t^{5/4}}
    \ge m \ge k
    \)
    and we also have
    \(
    k\ge \eps m \ge \frac{sm}{\sqrt{t'}} \ge \frac{\abs{V(R_2)}}{\sqrt{t'}}
    \).
    The lemma gives a decomposition
    $(\KF_1, \ldots, \KF_k, R_3)$ of $R_2$ such that: the $\KF_i$'s are isomorphic $K_{t',t'}$-forests,
    and 
    \(
    e(R_3)
    \le 
    \frac{4\abs{V(R_2)}^2}{\sqrt{t'}}
    \le \frac{4s^2 m^2}{\sqrt{t'}}
    \le \eps^2 m^2
    \) (using that $c=10^6$ and $t'\geq cs^4\eps^{-4}$).
    Let $\KF$ be the isomorphism class of the $\KF_i$'s. By moving fewer than $r$ components from each $\KF_i$ to \(R_3\), we may assume that $\KF$ is $r$-divisible;
    this increases the size of \(R_3\) by at most
    \(r (t')^2 k = r t k \le rtm \leq \eps^2 m^2\), so \(R_3\) has a total size at most \(2\eps^2 m^2\).
	
	For $x \in [s]$, let $c_x$ be the number of components of size $x$ in $\SF$.
    Notice that 
    \begin{equation}\label{eq:VKFbound}
		\abs{V(\KF)} = 2t'\frac{e(\KF)}{(t')^2}=\frac{2e(\KF)}{t'} 
		\le \frac{2m^2}{kt'}
		\le \frac{2m}{\eps t'}
		\le \frac{\eps^2m}{s^2}.
    \end{equation}
    If \(c_x< \eps^2m /s^2\), set \(c_x' := 0\).
    Otherwise, take $c_x'$ to be the largest integer which is divisible by $r$ and is at most \(c_x - \eps^2m /s^2\).
    Observe $c_x' \ge c_x - 2\eps^2m /s^2$,
    since for \(m\) sufficiently large, 
    \(\eps^2 m/s^2 \ge r\) and hence there is a multiple of \(r\) in the interval 
    \([c_x - 2\eps^2m /s^2, c_x - \eps^2m /s^2]\).
    Moreover, if \(c_x \ge \eps^2m /s^2\), from the above and (\ref{eq:VKFbound}) it follows that 
    \(c_x - c_x' \ge \abs{V(\KF)}\).
    Let $\SF'$ be the star forest which has $c_x'$ components of size $x$, 
    for each $x \in [s]$ (and has no stars larger than $s$), and let $H = \SF' + \KF$.
    Observe that since \(c_x'\) is a multiple of \(r\),  \(\SF'\) is  \(r\)-divisible, and so \(H\) is also \(r\)-divisible.

	Notice that, for every $i \in [k]$, the union $\SF_i \cup \KF_i$ contains a copy of $H$. Indeed, for every $x \in [s]$, remove $c_x - c_x'$ components of size $x$ from $\SF_i$, so that all such components that intersect $\KF_i$ are removed; this is possible because $c_x - c_x'$ is either the number of components of $\SF_i$ of size $x$, or at least as large as the number of vertices in $\KF_i$.
    Let $H_i$ be a copy of $H$ in $\SF_i \cup \KF_i$, for $i \in [k]$, and let $R$ be the graph spanned by edges not covered by the $H_i$'s.
	
	Then the $H_i$'s are isomorphic, \(r\)-divisible, and have components of size at most $t$, since in \(\KF_i\) each component has size  \(t=(t')^2\) and in \(\SF_i\) they have size at most \(s\le t\).
    The edges in $R$ consist of: edges of $R_1 \cup R_3$, of which there are at most $3\eps^2m^2$; 
    and edges in $\SF_i - H_i$.
    For the latter, notice that the number of components of size \(x \in [s]\) we removed from \(\SF\) is 
    \(c_x - c_x' \le  2\eps^2 m s^{-2} \),
    so the number of edges of \(\bigcup_{i=1}^k (\SF_i - H_i)\) is at most
    $ k\cdot s^2 \cdot 2\eps^2m s^{-2}=2\eps^2 mk \le 2\eps^2 m^2$. 
    Altogether, 
    $e(R) \le 5\eps^2 m^2 \le \eps m^2$.
    Thus $(H_1, \ldots, H_k, R)$ satisfies the requirements of the lemma.
\end{proof}

In the next lemma we further decompose the small remainder from the previous step into stars, obtaining a decomposition into isomorphic graphs, an ascending sequence of stars, and a remainder of small maximum degree, with the stars and isomorphic graphs interacting nicely.

To state the next lemma we need the following two definitions.
We say that a matching $M\subseteq G$ is \emph{isolated} if it touches no other edges of $G$. 
We say that an ordered  pair of graphs $(H_{i}, H_{i+1})$ is \emph{ascending} if  $H_i\cong H_{i+1}$ or $H_i\cong H_{i+1}\setminus e$ for some edge $e$ in $H_{i+1}$. 
We say that a sequence of graphs $(H_1, \dots, H_t)$ is ascending if $(H_i,H_{i+1})$ is ascending for each $i \in [t-1]$. In particular, all graphs in an ascending sequence can be isomorphic.

\begin{lemma}\label{lem:approx-isomorphic-decomposition-stronger}
Let \(\eps\in (0,\frac{1}{10})\), \(c:=10^6\),  \(m\) be large and let \(k\in [\eps m, m]\).
	Suppose that $G$ is a graph with $\Delta(G)\leq cm$ and $0.2m^2\leq e(G)\leq m^2$. Then there is a decomposition $(H_1, \ldots, H_k, S_1, \ldots, S_k, R)$ of $G$, such that: $H_1, \ldots, H_k$ are isomorphic, $2$-divisible, and contain isolated matchings of size at least $m/200c$; $(S_1, \dots, S_k)$ is an ascending sequence of stars with $e(S_i) = 0$ for $i \in [k - \eps m]$; $R$ satisfies $e(R)\leq \eps m^2$ and $\Delta(R)\leq \eps m$; and $H_i, S_i$ are vertex-disjoint for each $i\in [k]$.
\end{lemma}
\begin{proof}
	Let \(t:=\ceil{\eps^{-48}c^6}^2\) and $k' := \ceil{k/20}$. By choosing $m$ sufficiently large we can assume that it is larger than any fixed function of $c, t, \eps$.
	By Vizing's theorem, $G$ can be decomposed into $cm+1$ matchings, which we may assume are almost equal by Lemma~\ref{lem:balancing-matching}. 
    Then, by removing at most \(40\) edges from each such matching,
    there are edge-disjoint isomorphic $40$-divisible matchings $M_1, \ldots, M_{k'}$ of size at least $\frac{e(G) - 40(cm+1)}{cm+1} \ge m/10c$.

	Let $G'=G - \bigcup M_i$, and note \(\Delta(G') \le cm\).
	Apply Lemma~\ref{lem:approx-isomorphic-decomposition} (with $\eps_{\ref{lem:approx-isomorphic-decomposition}} = \eps^4$, $r_{\ref{lem:approx-isomorphic-decomposition}}=40$, $m_{\ref{lem:approx-isomorphic-decomposition}}=m$, $k_{\ref{lem:approx-isomorphic-decomposition}}=k'$) to obtain a decomposition $(H_1, \ldots, H_{k'}, R_1)$ of \(G'\), where the $H_i$'s are isomorphic, $40$-divisible, their components have size at most $t$, and $e(R_1) \le \eps^4 m^2$. Let $H$ and $M$ be the isomorphism classes of $H_i$ and $M_i$, respectively.

	Next we will use \Cref{lem:large-matching-forest} to decompose \(M\cup H\).
    Let \(K_1, \hdots, K_L\) be an enumeration of the components of \(H\).
    Then
    \(\sum_i \abs{K_i}^2  \le t \cdot \sum_i \abs{K_i} = t \abs{V(H)}\),
    and
    \(\abs{V(H)} \le e(H) \le e(G')/k' \le 20 \eps^{-1} m\), so the right hand side of \eqref{eq:large-matching-forest} is bounded by $O(\sqrt{m})$.
    On the other hand,
    $e(M)/5 \ge m/50c$,
	so the left hand side of \eqref{eq:large-matching-forest} is $\Omega(m)$, and hence \Cref{lem:large-matching-forest} applies and gives a decomposition of  $M\cup H$ into five copies of $H' = M/5 + H/5$.
    Let $H_1', \dots, H_{5k'}' \cong H'$ be the subgraphs of \(H_1\cup M_1, \dots, H_{k'} \cup M_{k'}\) resulting from applying this decomposition to each \(H_i\cup M_i\).
    Observe that $H'$ is $8$-divisible (since $H$ and $M$ are $40$-divisible), and  it contains an isolated matching of size at least \(e(M)/5 \ge  m/50c\).

    Next we decompose \(R_1\) into an ascending sequence of stars.
	Let $S_1, \ldots, S_r$ be a maximal sequence of edge-disjoint stars in $R_1$ with $e(S_i)=2i$, for $i \in [r]$. 
    Then $\eps^4 m^2 \geq e(R_1)\geq \sum_{i=1}^r e(S_i)=r(r+1)$ which gives $r\leq \eps^2 m$. 
	Let $R_2=R_1 \setminus \bigcup S_i$, noting that $\Delta(R_2)\leq 2r+1\leq 3\eps^2 m$.

	Finally, apply Lemma~\ref{lem:combine-star-with-forest} to $S_i, H_i'$, for each $i \in [r]$.
    This yields a decomposition $(H_i^1, \ldots, H_i^4, S_i^1, \ldots, S_i^4)$ of $S_i \cup H_i'$ where $e(S_i^1) = i$, $e(S_i^j) \leq i$ for $j \in [4]$;
    $H_i^j$ is isomorphic to $H'/4$, and thus is \(2\)-divisible and has an isolated matching of size at least \(m/200c\); and the graphs $H_i^j, S_i^j$ are vertex-disjoint. 
	For $i \in [r+1, 5k']$ let $(H_i^1, \ldots, H_i^4)$ be a decomposition of $H_i'$ into copies of $H'/4$ and let $S_i^j = \emptyset$ for $j \in [4]$ and $i \in [r+1, 5k']$.

	Relabelling, we obtain a decomposition $(H_1, \ldots, H_{20k'}, S_1, \ldots, S_{20k'})$ of $G \setminus R_2$, such that:
    the $S_i$'s are stars of size at most $r$, with at least one of them being a star of size exactly $i$ for every $i \in [r]$, and all but at most $4r \le \eps m$ of the stars are empty; 
    the $H_i$'s are isomorphic, $2$-divisible and contain isolated matchings of size at least $m/200c$; and $H_i, S_i$ are vertex-disjoint for $i \in [20k']$.
	Remove $20k' - k \le 20$ pairs $H_i, S_i$ with $S_i$ empty and possibly relabel, to obtain a sequence $(H_1, \ldots, H_k, S_1, \ldots, S_k)$ with the same properties, assuming additionally that $e(S_1) \le \ldots \le e(S_{k})$. 

	Let $R = G \setminus \bigcup_{i \in [k]}(S_i \cup H_i)$. Then $R$ is the union of $R_2$ and up to $20$ copies of $H_i$.
    We have \(e(H_i) \le m^2/k \le m/\eps\).
    Thus $e(R) \le e(R_2) + 20m/\eps \le e(R_1) + 20m/\eps \le \eps m^2$, and $\Delta(R) \le \Delta(R_2) + 20t \le \eps m$, using that each one of the removed \(H_i\) has components of size at most \(t\).
\end{proof}

By carefully combining the graphs in the decomposition given by \Cref{lem:approx-isomorphic-decomposition-stronger},
the next lemma gives an approximate ascending decomposition of every graph with linear maximum degree.

\begin{lemma}\label{lem:approx-ascending-decomposition}
	Fix $c=10^6$ and
    \(\delta \le 1/5000c\).
    Let \(\eps>0\) be sufficiently smaller than \(\delta\), and \(m\) be large enough.
    Set $b = \floor{\delta m}$.
	Suppose that $G$ is a graph with $\Delta(G)\leq cm$ and $m^2(1/2 - \eps) \le e(G) \le m^2(1/2 + \eps)$. 
	Then, there is a decomposition $(H_{b+1}, \dots, H_{m}, R)$ of $G$, such that: the sequence $(H_{b+1}, \dots, H_{m})$ is ascending; $H_{i}$ is a matching for $i \le m/2000c$; $e(H_{b+1}) \le b + b^2/m$; $e(H_{m})\geq m+1$;  and $\Delta(R)\leq b/10$.

	Additionally, for any $t\in[b+1, m]$, we can ensure that $e(H_t)=e(H_{t+1})$.
\end{lemma}
\begin{proof}
	Let 
    $\ell \in \{\floor{\eps m}, \floor{\eps m}+1\}$ 
    be such that $m - b - \ell$ is even, and write $k = \frac{m-b-\ell}2$. 
    Pick numbers $t_1\in [k-1]$ and 
    $t_2\in [\ell]$ 
    so that 
    $t-b\in\{t_1, k, 2k-t_1, 2k, 2k+t_2\}$ 
	(this is possible because $1 \le t-b \le m-b = 2k + \ell$).
    We will find a sequence of graphs as in the lemma so that $H_{i}\cong H_{i+1}$ for all $i$ with  
    $i-b\in \{t_1,k, 2k-t_1, 2k, 2k+t_2\}$,
    so in particular \(H_{t} \cong H_{t+1}\).

	By Lemma~\ref{lem:approx-isomorphic-decomposition-stronger} (with $k_{\ref{lem:approx-isomorphic-decomposition-stronger}} = k+\ell$ and $\eps_{\ref{lem:approx-isomorphic-decomposition-stronger}} = \eps^2$) there is a decomposition of $G$
	\begin{equation*}
		(H_1, \ldots, H_{k+\ell}, S_1, \ldots, S_{k+\ell}, R)
	\end{equation*}
    such that: the $H_i$'s are isomorphic, $2$-divisible, and contain isolated matchings of size at least $m/500c$; 
    the $S_i$'s form an ascending sequence of stars, with $e(S_i) = 0$ for $i \in [k+1]$ (since \(\ell - \eps^2 m \ge 1\) and thus $k+\ell-\eps^2 m \ge k+1$); 
    $S_i$ and $H_i$ are vertex-disjoint for each $i$; and $R$ satisfies $e(R) \le \eps^2 m^2$ and $\Delta(R) \le \eps^2 m$. Let $H$ be the isomorphism class of the $H_i$'s.

	If $e(S_{k+t_2+1})\neq e(S_{k+t_2})$, move an edge from each of $S_{k+t_2+1}, \dots, S_{k+\ell}$ to $R$ (so now $e(S_{k+t_2+1})= e(S_{k+t_2})$, $e(R) \le 2\eps m^2$, and $\Delta(R)\leq 2\eps m$).

	Let $h=e(H)/2$, noting that this is an integer (since $H$ is $2$-divisible), and set $a=h-k-b$.
	\begin{innerclaim} \label{claim}
		$(1/3) \cdot b^2/m \le a \le (2/3)\cdot b^2/m$.
	\end{innerclaim}
	\begin{proof}
		Write $O(\eps)$ to denote an expression which is in the interval $[-A\eps, A\eps]$, where $A$ is a constant that does not depend on $c, \delta, \eps, m$.
		Then $b = m \cdot (\delta + O(\eps))$, $k = \frac{m}{2}(1 - \delta + O(\eps))$,
		$\ell = O(\eps m)$,
        \(e(R) = O(\eps m^2)\)
        and \(\sum_{i=1}^{k+\ell} e(S_i) \le \ell^2 = O(\eps m^2)\) (because $e(S_{k+1})=0$ and $S_i$ is ascending).
        Thus,
		\begingroup
		\addtolength{\jot}{.4em}
		\begin{align*}
			a&= \frac{e(G)-e(R)-\sum e(S_i)}{2(k+\ell)}-k-b \\ 
			&= \frac{e(G) - e(R) - \sum e(S_i) - 2k^2 - 2k\ell - 2bk - 2b\ell}{2(k+\ell)} \\
			&= \frac{\left(\frac{1}{2} - \frac{1}{2}(1 - 2\delta + \delta^2) - (\delta - \delta^2) + O(\eps)\right)m^2}{(1 - \delta + O(\eps))m} \\
			&=\left(\frac{\delta^2}{2} + O(\eps)\right)
            \left(1+\delta + O(\delta^2)\right)m\\
			&= \frac{\delta^2m}{2} + O(\delta^3 m)
			= \frac{b^2}{2m} + O(\delta^3m).
		\end{align*}
		\endgroup
		In particular, $(1/3) \cdot b^2/m \le a \le (2/3) \cdot b^2/m$.
	\end{proof}

	Order $E(H)$ as $e(1), \dots, e(2h)$ so that $\{e(1), \dots, e({m/1000c})\}$ is an isolated matching and the prefix $\{e(1), \dots, e({h})\}$ is isomorphic to the suffix $\{e({h+1}),\dots, e({2h})\}$ (this is possible since $H$ is $2$-divisible and contains a matching of size $m/500c$).
	Let $e_i(j)$ be the copy of $e(j)$ in $H_i$.
	For $i \in [k]$, set 
	\begin{align*}
		& A_i^+ = \{e_i(1), \ldots, e_i(a+b+i+1)\} \\
		& A_i^- = \{e_i(1), \ldots, e_i(a+b+i)\} \\
		& B_i^+ = \{e_i(a+b+i+2), \ldots, e_i(2h)\} \\
		& B_i^- = \{e_i(a+b+i+1), \ldots, e_i(2h)\},
	\end{align*}
	and, for $i \in [k+1, k+\ell]$, define
	\begin{equation*}
		C_i = \{e_i(a+b+3), \ldots, e_i(2h)\} \cup S_i.
	\end{equation*}
    Observe that \(A_i^+\) and \(A_i^-\) are matchings for \(i\in[ m/2000c]\), since \(a+b+m/2000c +1 \le m/1000c\) and \(e(1),\hdots, e(m/1000c)\) is a matching.
	Moreover, note that $A_i^+\cong A_{i+1}^{+}\setminus\{e_{i+1}(a+b+i+2)\}$, showing that the sequence $(A_1^+, \ldots, A_k^+)$ is ascending. Similarly, the sequences $(A_1^-, \ldots, A_k^-)$, $(B_k^+, \ldots, B_1^+)$, and $(B_k^-, \ldots, B_1^-)$ are all ascending. The sequence $(C_{k+1}, \ldots, C_{k+\ell})$ is also ascending because the $S_i$'s are ascending.
	Additionally, note that: $A_i^+ \cong A_{i+1}^-$ and $B_i^+ \cong B_{i+1}^-$ for $i \in [k-1]$; $A_k^- \cong B_k^-$ (since $a+b+k=h$); and $B_1^+ \cong C_{k+1}$ (using $e(S_{k+1}) = 0$).
    Also, since \(e(S_{k+t_2}) = e(S_{k+t_2+1})\), we have \(C_{k+t_2} \cong C_{k+t_2+1}\).
	Altogether, this shows that the following sequence is ascending.
	\begin{equation*}
		A_1^+, \ldots, A_{t_1}^+, A_{t_1+1}^-, \ldots, A_k^-, B_k^-, \ldots, B_{t_1+1}^-, B_{t_1}^+, \ldots, B_1^+, C_{k+1}, \ldots, C_{k+\ell}.
	\end{equation*}
	Letting $F_{b+i}$ be the $i^{\text{th}}$ graph in this sequence, we get a sequence $F_{b+1}, \ldots, F_{m}$ which is ascending with $F_{b+i+1} \cong F_{b+i}$ for $i \in \{t_1, k, 2k-t_1, 2k, 2k+t_2\}$, where 
    \(F_{b+2k+t_2} \cong F_{b+2k+t_2+1}\) follows from \(C_{k+t_2} \cong C_{k+t_2+1}\). See Figure~\ref{Figure_triangle} for an illustration of this sequence.

	\begin{figure}
		\centering
		\includegraphics[width=1\textwidth]{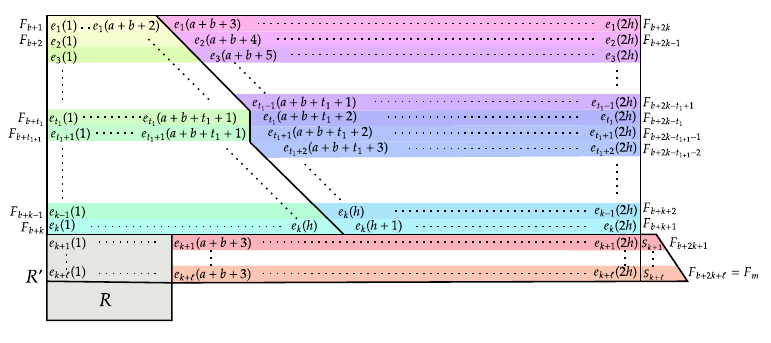}
		\caption{An illustration of how the graph $G$ is partitioned into $(F_{b+1}, \dots, F_m, R')$. Here all the edges of $G$ are pictured in the way we get them from \Cref{lem:approx-isomorphic-decomposition-stronger} at the start of the proof --- either as edges $e_i(j)$ of $H_i$, the stars $S_i$, or the remainder $R$. The isomorphic graphs $H_i$ are arranged in a rectangle so that any subset of edges $e_i(j)$ in row $i$ is isomorphic to the corresponding subset on any other row. The coloured areas represent the final partition $(F_b, \dots, F_m, R')$ which ends up satisfying the lemma.} 
		\label{Figure_triangle}
	\end{figure}

	Since \(A_i^+, A_i^-\) are matchings for \(i\in[ m/2000c]\) we have that $F_{b+i}$ is a matching for $ i \le m/2000c$. We have $e(F_{b+1}) = a+b+2 \le b + b^2/m$, by \Cref{claim}. Since $F_{b+1}, \ldots, F_{m}$ is ascending and $e(F_{b+i}) = e(F_{b+i+1})$ for at most $5 + \ell$ values of $i$ (at worst, they are not strictly ascending for  $i\in \{t_1, k, 2k-t_1, 2k, 2k+t_2\}\cup [k+1, k+\ell]$), we have that 
	\begin{equation*}
		e(F_m) \ge e(F_{b+1}) + m - b - 1 - (5 + \ell) =(a+b+2)+m-b-\ell-6=m + a - \ell - 4 \ge m + 1, 
	\end{equation*}
	using that $a=\Theta(\delta^2m)$ (by \Cref{claim}), $\ell=O(\eps m)$ and that $\eps$ is sufficiently small compared to $\delta$.

	Write $R' = R \cup \bigcup_{i = k+1}^{k+\ell} \{e_i(1), \ldots, e_i(a+b+2)\}$. Then $\Delta(R') \le \Delta(R) + \ell \le 4\eps m \le b / 10$ (using that $\{e_i(1), \ldots, e_i(a+b+2)\}$ is a matching for every $i$).
	Noting that $\{A_i^-, B_i^-\}$ and $\{A_i^+, B_i^+\}$ are decompositions of $H_i$, it follows that $(F_{b+1}, \ldots, F_m, R')$ is a decomposition of $G$, satisfying the requirements of the lemma.
\end{proof}

For the proof of \Cref{lem:max-deg-linear} we will need the following simple observation.
\begin{observation}\label{obs:ascending-after-deleted-matchings}
	Let $(H_i, H_{i+1})$ be ascending and $M_i\subseteq H_i, M_{i+1}\subseteq H_{i+1}$ be isolated matchings. Suppose we have one of
	\begin{itemize}
		\item 
			$e(H_{i+1}\setminus M_{i+1})=e(H_i\setminus M_i)+1$.
		\item  
			$e(H_{i+1}\setminus M_{i+1})=e(H_i\setminus M_i)$ and $H_{i+1} \cong H_i$.
	\end{itemize}
	Then $(H_{i}\setminus M_{i}, H_{i+1}\setminus M_{i+1})$ is ascending.
\end{observation}
\begin{proof}
	Suppose \(H_{i+1} \cong H_i\).
	If the first condition holds, we must have \(\abs{M_{i}} = \abs{M_{i+1}}+1\).
	Hence, since \(M_i, M_{i+1}\) do not intersect any other edges of \(H_i, H_{i+1}\) respectively, \((H_i\setminus M_i, H_{i+1}\setminus M_{i+1})\) is ascending.
	If the second condition holds then 
	\(\abs{M_i} = \abs{M_{i+1}}\) so 
	\(H_{i+1}\setminus M_{i+1} \cong H_i\setminus M_i\).

	Suppose \((H_i, H_{i+1})\) is ascending with \(e(H_{i+1}) = e(H_i) +1\).
	Then only the first condition can hold, and it implies that \(\abs{M_i} = \abs{M_{i+1}}\).
	Hence, since \(M_i, M_{i+1}\) are isolated, we have  that \((H_i\setminus M_i, H_{i+1}\setminus M_{i+1})\) is ascending.	
\end{proof}

Having finished all the necessary preparations, we are now ready to prove Lemma~\ref{lem:max-deg-linear}. This lemma is stated in \Cref{subsec:reduction}, where we use it to prove the main result of this paper, Theorem~\ref{thm:main}.
It says that for $c=10^6$ and \(m\) sufficiently large, if $G$ is a graph satisfying $e(G) \in (\binom{m}{2}, \binom{m+1}{2}]$ and $\Delta(G) \le cm$, then $G$ has an ascending subgraph decomposition.

\begin{proof}[Proof of Lemma~\ref{lem:max-deg-linear}]
    Let \(\delta  = 10^{-7}c^{-1}=10^{-13}\), let \(\eps\) be sufficiently smaller than \(\delta\) so that \Cref{lem:approx-ascending-decomposition} applies,
    and let \(b = \floor{\delta m}\).
	Let $t=e(G)-\binom{m}{2}$, noting that $1 \le t \le m$. 
	Write 
	\begin{equation*}
		e_i = \left\{
			\begin{array}{ll}
				i & i \in [t]\\
				i - 1 & i \in [t+1, m].
			\end{array}
			\right.
	\end{equation*}
	Apply Lemma~\ref{lem:approx-ascending-decomposition} to get an ascending sequence $(H_{b+1}, \dots, H_{m})$ such that $H_{i}$ is a matching for $i \in [b+1,  m/2000c]$;
    $e(H_{b+1}) \le b + b^2/m$; $e(H_m)\geq m+1$; the graph $R := G \setminus \bigcup H_i$ has maximum degree at most $b/10$; and, if \(t \in [b+1,m]\), we moreover require the decomposition is such that \(H_t \cong H_{t+1}\).

	Set $x_i:=e(H_i)-e_i$, \(i \in [b+1,m]\).
	We claim that $0 < x_i \le 2b^2/m$ for all \(i \in [b+1,m]\). Indeed, if we ever had $e(H_i)\leq e_i$, then we would have $e(H_i) \le i$ and, since the sequence is ascending, we would also have $e(H_m) \le m$, a contradiction. For the upper bound, using that $H_{b+1}, \ldots, H_m$ is ascending and the definition of $e_i$, we get
	\begin{align*}
		x_i - x_{b+1} = e(H_i) - e(H_{b+1}) - (e_i - e_{b+1})
		\le i - (b+1) - \big((i-1)  - (b+1)\big) = 1,
	\end{align*}
	showing that $x_i \le x_{b+1} + 1= e(H_{b+1})-e_{b+1}+1\le (b+b^2/m)-b+1\le 2b^2/m$.
	Write $a = \max x_i$ (so $1 \le a \le 2b^2/m$).

	Randomly pick an isolated matching $M_i\subseteq H_i$ of size $x_i$, making the choices independently for $i \in [b+1, m]$. There is always room to pick such a matching since for $i \le m/2000c$, $H_i$ is a matching of size $e(H_i)\geq i\geq b\geq 2b^2/m\geq a$, while for $i>m/2000c$ the graph $H_i$ contains an isolated matching of size $m/2000c\geq 2b^2/m\geq a$.
	This also shows that there are at least $i$ choices for each edge of $M_i$ for $i \le m/2000c$, and there are at least $m/2000c$ choices for each edge for $i > m/2000c$. Thus,
	\begin{align*}
		& \prob{v \in V(M_i)} 
            \le
		\left\{
			\begin{array}{ll}
				\frac{a}{i} & \text{for $i \le m/2000c$}\\[.4em]
				\frac{2000ca}{m} & \text{for $i > m/2000c$}.
			\end{array}
		\right.
	\end{align*}
	Letting $F=\bigcup_i M_i$, we have 
	\begingroup
	\addtolength{\jot}{.4em}
	\begin{align*}
		\expect{d_F(v)}
		& \leq \sum_{b+1 \,\le\, i \,\le\, m/2000c}\frac{a}{i} + \sum_{m/2000c < i \le m}\frac{2000ca}{m} \\
		& \leq a\int_{x=b}^{m/2000c}\frac1x \, dx + 2000ca \\
		& =  a\ln\left(\frac{m}{2000cb}\right)+ 2000ca \\
		& \leq a\ln(m/b)+ 2000ca\leq b/100,
	\end{align*}
	\endgroup
	using 
	$\ln(m/b)\leq m/ 1600 b\leq b/800a$
	and $2000ca \le 4000c b^2/m \le b/200$ since 
	$b/m \le \delta  = 10 ^{-7}c^{-1}$. Since the number of $H_i$ is $m-b$,
	by Chernoff's bound (\Cref{thm:Chernoff}), for each vertex \(v\) we have 
	$\prob{d_F(v)\geq b/50} \le e^{-2(b/100)^2/(m-b)} \le e^{-\delta^2m /10000}$. 
	Taking the union bound over all (non-isolated) vertices $v$ in $G$ (of which there are at most $2e(G)\leq 2m^2$), we deduce that, with positive probability, for all vertices \(v\), \(d_F(v)\leq b/50\).
	Hence there are matchings $M_{b+1},\hdots, M_m$ so that $\Delta(F)\leq b/50$.

	By definition of $x_i, M_i$ we have $e(H_i\setminus M_i)=e_i$ for $i \in [b+1, m]$,  and hence
	$e(F\cup R)= e_1+\dots+ e_{b}=\binom{b+1}{2}$ (using that $e(G)=\sum_{i=1}^me_i$). 
	We also have $\Delta(F \cup R) \leq b/5$. 
	Therefore, by Lemma~\ref{lem:asd-bounded-degree}, $F \cup R$ has an ASD into matchings $(M_1, \dots, M_b)$ with $e(M_i)=e_i$.
	Now $(M_1, \dots, M_b, H_{b+1}\setminus M_{b+1}, \dots, H_{m}\setminus M_m)$ is an ascending decomposition of $G$: first note that for each $i\leq b$, we have $e(M_i)=e_i$ and for $i\geq b+1$ we have $e(H_i\setminus M_i)=e_i$. Next, notice that the graphs $M_1, \dots, M_{b}, H_{b+1}\setminus M_{b+1}$ are matchings, and so this sequence is ascending. 
	Finally, it follows from \Cref{obs:ascending-after-deleted-matchings} that $H_{b+1} \setminus M_{b+1}, \ldots, H_m \setminus M_m$ is ascending: indeed, for \(i\neq t\), the first condition of \Cref{obs:ascending-after-deleted-matchings} applies to 
    \((H_{i}\setminus M_i, H_{i+1}\setminus M_{i+1})\), and for \(i=t\), if \(t\ge b+1\) the second condition applies, since \(H_t \cong H_{t+1}\).
    For \(t\le b\), we have \(M_t \cong M_{t+1}\).
	This completes the proof of the lemma.
\end{proof}    

\section{Conclusion} \label{sec:conclusion}
We proved the main conjecture of~\cite{original87} by showing that each graph has an ascending subgraph decomposition consisting of star forests and subgraphs of $K_{t,t}$'s.
It would be interesting to understand whether the latter graphs are necessary or just an artefact of our proof. Faudree, Gy\'arf\'as, and Schelp have conjectured that ascending decomposition purely using star forests should always exist.
\begin{conjecture}[Faudree, Gy\'arf\'as, and Schelp \cite{faudree87}]\label{Conj_star_forests}
    Every graph $G$ with $\binom{m+1}2$ edges has an ascending subgraph decomposition $H_1, \dots, H_m$, where each $H_i$ is a star forest.    
\end{conjecture}
The techniques introduced in this paper are likely to be useful for approaching this conjecture (for large $m$). However, we think that new ingredients would be needed too --- mainly because one of our key intermediate results (\Cref{lem:approx-isomorphic-decomposition}) is not true when one restricts to star-forest decompositions. Indeed, the essence of \Cref{lem:approx-isomorphic-decomposition} is that every graph with $\binom m2$ edges can be nearly-decomposed into $\eps m$ isomorphic subgraphs (for $\eps^{-1}\ll m$). But the complete graph $K_m$ cannot be nearly-decomposed into $\eps m$
isomorphic star forests for $\eps<1/2$ (just because any star forest in $K_m$ has at most $m-1$ edges). Thus it seems necessary to deviate from our proof strategy if one wants to prove \Cref{Conj_star_forests} (at least when $G$ is very dense).

It is possible that even stronger generalizations of the ascending subgraph decomposition conjecture are true. Another feature of our proof of \Cref{thm:main} is that it produced a decomposition into graph $H_i$ which are all the disjoint union of one potentially large star and a lot of connected components of bounded size. If these small components could be made to have size $1$, then we would obtain a strengthening of \Cref{Conj_star_forests}.
\begin{problem} 
Does every sufficiently large graph $G$ with $\binom{m+1}2$ edges have an ascending subgraph decomposition $H_1, \dots, H_m$, where each $H_i$ is a disjoint union of a star and a matching?  
\end{problem}

Igor Balla suggested the following variant of our problem: for which sequences $a_1 \le \ldots \le a_m$ does every graph on $a_1 + \ldots + a_m$ edges have a decomposition $(H_1, \ldots, H_m)$ with $e(H_i) = a_i$ and $H_i$ isomorphic to a subgraph of $H_{i+1}$?
Our methods might work if $a_i + a_{m-i}$ is the same for all $i$.

Recall that one of the central open problems in the area of graph decompositions is the Gy\'arf\'as tree packing conjecture: if $T_1, \ldots, T_{n-1}$ is any sequence of trees with $e(T_i) = i$, then we can decompose $K_n$ into copies of \(T_1,\hdots,T_{n-1}\).
This is only known when $\Delta(T_i') \le \frac{cn}{\log n}$~\cite{allen2021}.
Nati Linial asked whether this becomes easier if we assume that the sequence of trees $T_1, \ldots, T_{n-1}$ is ascending.

Finally,
clearly for a star forest to have a star ASD,
the $i^{\text{th}}$ smallest component needs to have size at least $i$. This shows that our condition in Theorem~\ref{thm:main-stars} is tight up to constant factors. 
It would be interesting to determine the precise constants that are necessary.

\bibliographystyle{amsplain} 
\bibliography{bibliography}   
\end{document}